\documentclass[draft,a4paper]{amsart}

\usepackage{amsfonts,amssymb,amscd,amsmath,latexsym,amsbsy,enumerate,stmaryrd,a4wide,verbatim}
\usepackage{hyperref} 

\theoremstyle{plain}
\newtheorem{thm}{Theorem}[section]
\newtheorem{cor}[thm]{Corollary}

\newtheorem{lem}[thm]{Lemma}
\newtheorem{prop}[thm]{Proposition}

\theoremstyle{remark}
\newtheorem{rem}[thm]{Remark}
\numberwithin{equation}{section}

\newcommand{\R}{\mathbb R}
\newcommand{\N}{\mathbb N}
\newcommand{\C}{\mathbb C}
\newcommand{\Z}{\mathbb Z}

\newcommand{\V}{\mathcal V}

\newcommand{\al}{\alpha}
\newcommand{\be}{\beta}
\newcommand{\ga}{\gamma}
\newcommand{\Ga}{\Gamma}
\newcommand{\de}{\delta}

\newcommand{\eps}{\varepsilon}
\newcommand{\si}{\sigma}
\newcommand{\te}{\theta}

\newcommand{\la}{\lambda}
\newcommand{\La}{\Lambda}
\newcommand{\ka}{\kappa}

\newcommand{\Om}{\Omega}

\newcommand{\rFs}[5]{\,_{#1}F_{#2} \left( \genfrac{.}{.}{0pt}{}{#3}{#4}
\ ;#5 \right)}

\newcommand{\Res}[1]{\underset{#1}{\mathrm{Res}}}
\newcommand{\lc}{\mathrm{lc}}

\begin{document}
\title[A hypergeometric function transform]{A hypergeometric function transform and matrix-valued orthogonal polynomials}
\author{Wolter Groenevelt}
\author{Erik Koelink}
\address{Technische Universiteit Delft, DIAM, PO Box 5031,
2600 GA Delft, the Netherlands}
\email{w.g.m.groenevelt@tudelft.nl}
\address{Radboud Universiteit, IMAPP, FNWI, Heyendaalseweg 135, 6525 AJ Nijmegen,
the Netherlands}
\email{e.koelink@math.ru.nl}
\maketitle

\begin{abstract}
The spectral decomposition for an explicit second-order differential operator $T$ is determined.
The spectrum consists of a continuous part with multiplicity two, a continuous part with
multiplicity one, and a finite discrete part with multiplicity one. The spectral analysis gives
rise to a generalized Fourier transform with an explicit hypergeometric function as a kernel.
Using Jacobi polynomials the operator $T$ can also be realized as a five-diagonal operator, hence
leading to orthogonality relations for $2\times 2$-matrix-valued polynomials. These matrix-valued
polynomials can be considered as matrix-valued generalizations of Wilson polynomials.
\end{abstract}

\section{Introduction}
It is well-known that a three-term recurrence relation
\[
\la p_n(\la) = a_n p_{n+1}(\la) + b_n p_{n}(\la) + a_{n-1} p_{n-1}(\la),\qquad  n =0,1,2,\ldots,
\]
with $a_{-1}=0$, can be solved using orthogonal polynomials. A generalization of this is obtained
by Dur\'an and Van Assche \cite{DuraVA}, who showed a $2N+1$-term recurrence relation can be solved
using $N\times N$-matrix-valued orthogonal polynomials. Motivated by this result and previous work
by Ismail and the second author \cite{IK}, \cite{IK1}, a method is presented by Ismail and the
authors \cite{GIK} to obtain orthogonality relations for $2\times 2$-matrix-valued orthogonal
polynomials from an operator $T$ on a Hilbert space $\mathcal H$ of functions. The operator $T$
must satisfy the following conditions:
\begin{enumerate}[(i)]
\item \label{conditioni}$T$ is self-adjoint;
\item \label{condtionii}there exists a weighted Hilbert space $L^2(\V)$ and a unitary operator $U:\mathcal H \to L^2(\V)$ so that $UT=MU$, where $M$ is the multiplication operator on $L^2(\V)$;
\item \label{condtioniii}there exists an orthonormal basis $\{f_n\}_{n=0}^\infty$ of $\mathcal H$,
and there exist sequences $(a_n)_{n=0}^\infty$, $(b_n)_{n=0}^\infty$, $(c_n)_{n=0}^\infty$ of
numbers with $a_n>0$ and $c_n \in \R$ for all $n \in \N$, such that
    \[
    T f_n = a_n f_{n+2} + b_n f_{n+1} + c_n f_n + \overline{b_{n-1}} f_{n-1} + a_{n-2}f_{n-2},
    \]
    where we assume $a_{-1}=a_{-2}=b_{-1}=0$.
\end{enumerate}
In \cite{GIK} two explicit examples are worked out, where the operator $T$ is, besides a five-term
operator, also realized as the second-order $q$-difference operator corresponding to well-known
$q$-hypergeometric orthogonal polynomials. Thus, the unitary operator $U$ is the integral transform
with the corresponding orthogonal polynomials as a kernel. This leads to complicated, but explicit,
orthogonality relations for certain matrix-valued polynomials defined by an explicit matrix three-
term recurrence relation.
We note that the explicit weight function differs structurally from the usually considered weight
functions for matrix-valued orthogonal polynomials consisting of a matrix-deformation of a
classical weight.

In this paper we apply the method from \cite{GIK} with the second-order differential operator
$T=T^{(\al,\be;\ka)}$ defined by
\begin{equation} \label{eq:T}
T = (1-x^2)^2 \frac{d^2}{dx^2} + (1-x^2)[\be-\al-(\al+\be+4)x]\frac{d}{dx} +
\frac14[\ka^2-(\al+\be+3)^2] (1-x^2).
\end{equation}
Here $\al,\be > -1$ and $\ka \in \R_{\geq 0} \cup i\R_{>0}$. The differential operator $T$ is
closely related to the second-order differential operator to which the Jacobi polynomials are
eigenfunctions. It should be noted that $T$ raises the degree of a polynomial by $2$, so there are
no polynomial eigenfunctions.
We will show that the differential operator $T$, considered as an unbounded operator on a weighted
$L^2$-space, satisfies conditions (\ref{conditioni})--(\ref{condtioniii}) given above. An
interesting problem here is that $T$ does not correspond to orthogonal polynomials or to a known
unitary integral transform such as the Jacobi function transform \cite{K}.

The unitary operator $U$ needed in condition (\ref{condtionii}) is given by an explicit integral
transform $\mathcal F$ which is obtained from spectral analysis of $T$. The spectrum of $T$
consists of a continuous part with multiplicity two, a continuous part with multiplicity one, and a
(possibly empty) finite discrete part of multiplicity one.
As a result, the integral transform $\mathcal F$ has a
hypergeometric kernel which is partly $\C^2$-valued and partly $\C$-valued. There are several (but
not very many) hypergeometric integral transforms with $\C^2$-valued kernels available in the
literature, see e.g.~\cite{N}, \cite{GKR}, \cite[Exercise (4.4.11)]{Koel}, see also \cite{Groen}
for an example with basic hypergeometric functions. To the best of our knowledge all known examples
can be considered as nonpolynomial extensions of hypergeometric orthogonal polynomials, in the
sense that the corresponding kernels are eigenfunctions of a differential/difference operator that
also has orthogonal polynomials as eigenfunctions. For example, Neretin's $\C^2$-valued $_2F_1$-
integral transform \cite{N} generalizes the Jacobi polynomials. The integral transform $\mathcal F$
we consider in this paper, however, does not seem to generalize a family of orthogonal polynomials,
although in a special case it is can be considered as a nonpolynomial extension of two different
one-parameter families of Jacobi polynomials. Furthermore, other hypergeometric integral transforms
and hypergeometric orthogonal polynomials correspond to a bispectral problem, see
e.g.~\cite{Grunb}, which can always be related directly to contiguous relations for hypergeometric
functions. From the explicit expressions as hypergeometric functions for the kernel of
$\mathcal F$, it is unclear wether $\mathcal F$ is also related to a bispectral problem.

In a special case the $2\times 2$-matrix-valued orthogonal polynomials we obtain can be
diagonalized. In this case the orthogonality relations correspond to orthogonality relations for
two subfamilies of Wilson polynomials \cite{Wi80}. This is why we consider our matrix-valued
polynomials as generalizations of (subfamilies of) Wilson polynomials.

The organization of this paper is as follow. In Section \ref{sec:inttrans} we introduce the
integral transform $\mathcal F$ and show that the differential operator $T$ \eqref{eq:T} satisfies
conditions (\ref{conditioni}) and (\ref{condtionii}). The proofs for this section are given
separately in Section 4, where the spectral analysis of $T$ is carried out, which can be quite
technical at certain points. In Section \ref{sec:matrixvaluedpol} we realize $T$ as a five-diagonal
operator on a basis consisting of Jacobi polynomials, so that condition (\ref{condtioniii}) is also
satisfied. The corresponding five-term recurrence relation is equivalent to a matrix three-term
recurrence relation that defines $2\times 2$-matrix-valued orthogonal polynomials $P_n$ for which
the orthogonality relations are determined. We also consider briefly the special case $\al=\be$, in
which case the integral transform $\mathcal F$ reduces to two Jacobi function transforms and the
orthogonality relations for $P_n$ correspond to the orthogonality relations for certain Wilson
polynomials. Finally, in Section \ref{sec:proofs} eigenfunctions of $T$ are given, which are needed
for the spectral decomposition of $T$. The spectral decomposition leads to a proof of the unitarity
of the integral transform $\mathcal F$, and to an explicit formula for its inverse.\\

\textbf{Notations.} We write $\N$ for the set of nonnegative integers. We use standard notations
for hypergeometric functions, as in e.g.~\cite{AAR,Isma}. For products of $\Ga$-functions and of
shifted factorials we use the shorthand notations
\[
\begin{split}
\Ga(a_1,a_2,\ldots,a_n) & = \Ga(a_1)\Ga(a_2)\cdots \Ga(a_n),\\
(a_1,a_2,\ldots, a_n)_k & = (a_1)_k (a_2)_k \cdots (a_n)_k.
\end{split}
\]

\section{Spectral analysis and a hypergeometric function transform} \label{sec:inttrans}
In this section we describe the spectral analysis of the operator $T$ defined by \eqref{eq:T}. The
spectral decomposition is given by an integral transform with certain hypergeometric $_2F_1$-
functions as a kernel which is interesting in its own right. The proofs for this section are
postponed until Section \ref{sec:proofs}.\\

Let $\al,\be>-1$ be fixed, and let $w^{(\al,\be)}$ be the Jacobi weight function on $[-1,1]$ given by
\begin{equation} \label{eq:weightfunction}
w^{(\al,\be)}(x) = 2^{-\al-\be-1}\frac{\Ga(\al+\be+2)}{\Ga(\al+1,\be+1) }(1-x)^\al (1+x)^\be.
\end{equation}
The corresponding inner product is denoted by $\langle \cdot,\cdot \rangle$,
\[
\langle f,g \rangle  = \int_{-1}^1 f(x) \overline{g(x)} w^{(\al,\be)}(x)\,dx.
\]
The weight is normalized such that $\langle 1,1\rangle=1$.  We denote by
$\mathcal H=\mathcal H^{(\al,\be)}$ the corresponding weighted $L^2$-space;
$\mathcal H= L^2((-1,1),w^{(\al,\be)}(x)dx)$. To stress the dependence on the parameters
 $\al$ and $\be$, we will sometimes denote the inner product in $\mathcal H^{(\al,\be)}$ by
 $\langle\cdot,\cdot \rangle_{\al,\be}$. Let us remark that the substitution $x \mapsto -x$ sends
 $T^{(\al,\be;\ka)}$ to $T^{(\be,\al;\ka)}$, and $\mathcal H^{(\al,\be)}$ to $\mathcal H^{(\be,\al)}$.
 So without loss of generality we may assume $\be \geq \al$, which we do from here on.\\

We consider $T$ as an unbounded operator on $\mathcal H$. The domain $\mathcal D_0$ for $T$ is
described in Section \ref{ssec:SpectralAnalysis}, where the following result is proved.
\begin{prop}
The operator $(T,\mathcal D_0)$ has a unique self-adjoint extension.
\end{prop}

We denote the extension of $T$ again by $T$. The spectral analysis of $T$ will be described by the
integral transform $\mathcal F$ mapping functions in $\mathcal H$ (under suitable conditions) to
functions in the Hilbert space $L^2(\V)$. We first introduce the latter space.\\

Let $\Om_1, \Om_2 \subset \R$ be given by
\[
\Om_1=\big(-(\be+1)^2,-(\al+1)^2\big)\quad \text{and} \quad \Om_2=\big(-\infty,-(\be+1)^2\big).
\]
We set
\begin{equation} \label{eq:dela and etala}
\begin{split}
\de_\la = i \sqrt{ -\la-(\al+1)^2},&  \qquad \la \in \Om_1 \cup \Om_2,\\
\eta_\la = i \sqrt{ -\la-(\be+1)^2},& \qquad \la \in \Om_2,\\
\de(\la) = \sqrt{\la +(\al+1)^2 },& \qquad \la \in \C \setminus \big(\Om_1 \cup \Om_2\big),\\
\eta(\la) = \sqrt{\la+(\be+1)^2}, & \qquad \la \in \C \setminus \Om_2.
\end{split}
\end{equation}
Here $\sqrt{\cdot}$ denotes the principal branch of the square root.
For $n \in \N$, we define $\la_n \in \C$ as the solution of
\begin{equation} \label{def:la n}
\de(\la)+\eta(\la)=\sqrt{ \la + (\al+1)^2 } + \sqrt{\la+(\be+1)^2} = -2n-1 + \ka.
\end{equation}
We define the finite set $\Om_d$ by
\[
\Om_d = \{ \la_n \mid n \in \N \text{ and } n\leq \frac12(\ka-1)\},
\]
i.e., $\Om_d$ consists of the real solutions of \eqref{def:la n}. Note that $\Om_d = \emptyset$ if $\ka < 1$ or $\ka \in i\R_{>0}$. The number $\la_n \in \Om_d$ has the explicit expression
\[
\begin{split}
\la_n &= \left( -n+\frac12(\ka-1) + \frac{ (\al-\be)(\al+\be+2)}{-4n-2+2\ka}\right)^2 - (\al+1)^2 \\
& = \left( -n+\frac12(\ka-1) - \frac{ (\al-\be)(\al+\be+2)}{-4n-2+2\ka}\right)^2 - (\be+1)^2.
\end{split}
\]
We will denote by $\si$ the set $\Om_2 \cup \Om_1 \cup \Om_d$. Theorem \ref{thm:integraltransform} will show that $\si$ is the spectrum of $T$.

Next we introduce the weight functions that we need to define $L^2(\V)$. First we define
\begin{equation} \label{eq:c-function}
c(x;y) = \frac{\Ga(1+y,-x)}{\Ga(\frac12(1+y-x+\ka),\frac12(1+y-x-\ka))}.
\end{equation}
With this function we define for $\la \in \Om_1$
\begin{equation} \label{eq:v(1)}
v(\la) = \frac{1}{ c\big(\de_\la;\eta(\la)\big)c\big(-\de_\la;\eta(\la)\big) }.
\end{equation}
For $\la \in \Om_2$ we define the matrix-valued weight function $V(\la)$ by
\begin{equation} \label{eq:V}
V(\la)=\begin{pmatrix} 1 & v_{12}(\la) \\ v_{21}(\la) & 1 \end{pmatrix},
\end{equation}
with
\begin{equation} \label{eq:v(2)}
v_{21}(\la) = \frac{ c(\eta_\la;\delta_\la) }{c(-\eta_\la;\delta_\la)} =\frac{ \Ga\big(-\eta_\la,\frac12(1+\de_\la+\eta_\la+\ka),\frac12(1+\de_\la+\eta_\la-\ka)\big) }{\Ga\big(\eta_\la,\frac12(1+\de_\la-\eta_\la+\ka),\frac12(1+\de_\la- \eta_\la-\ka)\big)},
\end{equation}
and $v_{12}(\la) = \overline{v_{21}(\la)}$. Finally, for $\la_n \in \Om_d$ we set
\begin{equation} \label{eq:N la n}
\begin{split}
N_{\la_n} &= \Res{\la=\la_n}\left(\frac{c\big(\eta(\la);\de(\la)\big)}{\eta(\la)c\big(-\eta(\la);\de(\la)\big)}\right)\\
& = \frac{4\de(\la_n)}{-2n-1+\ka}\frac{(-1)^n \Ga\big(-\eta(\la_n),\ka-n\big) }{n!\,\Ga\big(\eta(\la_n), \frac12(1+\de(\la_n)-\eta(\la_n)+\ka), \frac12(1+\de(\la_n)-\eta(\la_n)-\ka)\big) }.
\end{split}
\end{equation}
Note here that $\de(\la_n)-\eta(\la_n) = \frac{(\al-\be)(\al+\be+2)}{-2n-1+\ka}$.

Now we are ready to define the Hilbert space $L^2(\V)$. It consists of functions that are $\C^2$-valued on $\Om_2$ and $\C$-valued on $\Om_1 \cup \Om_d$. The inner product on $L^2(\V)$ is given by
\[
\begin{split}
\langle f,g \rangle_{\V} &= \frac{1}{2\pi D} \int_{\Om_2} g(\la)^* V(\la)
f(\la) \frac{d\la}{-i\eta_\la} \\
& \quad + \frac{1}{2\pi D} \int_{\Om_1} f(\la) \overline{g(\la)} v(\la) \frac{d\la}{-i\de_\la}
+ \frac{1}{D} \sum_{\la \in \Om_d} f(\la) \overline{g(\la)} N_{\la},
\end{split}
\]
where $D=\dfrac{4\Ga(\al+\be+2)}{\Ga(\al+1,\be+1) }$.\\

Next we introduce the integral transform $\mathcal F$. For $\la \in \Om_1$ and $x \in (-1,1)$ we define
\begin{equation} \label{def:phi}
\begin{split}
\varphi_\la(x) &= \left(\frac{1-x}{2}\right)^{-\frac12(\al- \de_\la+1)} \left(\frac{1+x}{2}\right)^{-\frac12(\be - \eta(\la)+1)} \\ & \qquad \times \rFs{2}{1}{ \frac12(1 +\de_\la + \eta(\la)-\ka), \frac12(1+ \de_\la + \eta(\la)+\ka)}{1+ \eta(\la)}{\frac{1+x}{2}}.
\end{split}
\end{equation}
By Euler's transformation, see e.g.~\cite[(2.2.7)]{AAR}, we can replace $\de_\la$ by $-\de_\la$
in \eqref{def:phi}. Furthermore, we define for $\la \in \Om_2$ and $x\in (-1,1)$,
\begin{equation} \label{def:phi+-}
\begin{split}
\varphi^\pm_\la(x) &=\left(\frac{1-x}{2}\right)^{-\frac12(\al - \de_\la+1)} \left(\frac{1+x}{2}\right)^{-\frac12(\be \mp \eta_\la+1)} \\ & \qquad \times \rFs{2}{1}{ \frac12(1 + \de_\la \pm \eta_\la-\ka), \frac12(1 + \de_\la \pm \eta_\la+\ka)}{1\pm \eta_\la}{\frac{1+x}{2}}.
\end{split}
\end{equation}
Observe that $\overline{\varphi^+_\la(x)}=\varphi^-_\la(x)$,
again by Euler's transformation.
Finally, for $\la_n \in \Om_d$ we define
\begin{equation} \label{def:phi la n}
\varphi_{\la_n}(x) = \left(\frac{1-x}{2}\right)^{-\frac12(\al-\de(\la_n)+1)} \left(\frac{1+x}{2}\right)^{-\frac12(\be-\eta(\la_n)+1)} \rFs{2}{1}{-n,\kappa-n}{1+\eta(\la_n)}{\frac{1+x}{2}}.
\end{equation}
Now, let $\mathcal F$ be the integral transform defined by
\begin{equation} \label{def:F}
(\mathcal Ff)(\la) =
\begin{cases}
\displaystyle \int_{-1}^1 f(x) \begin{pmatrix} \varphi^+_\la(x)\\ \varphi^-_\la(x) \end{pmatrix} w^{(\al,\be)}(x) \,dx,& \la \in \Om_2,\\
\displaystyle \int_{-1}^1 f(x) \varphi_\la(x) w^{(\al,\be)}(x) \,dx,& \la \in \Om_1,\\
\displaystyle \int_{-1}^1 f(x) \varphi_{\la_n}(x) w^{(\al,\be)}(x) \,dx,& \la=\la_n \in \Om_d,
\end{cases}
\end{equation}
for all $f \in \mathcal H$ such that the integrals converge. The following result says that $\mathcal F$ is the required unitary operator $U$ from the introduction.

\begin{thm} \label{thm:integraltransform}
The integral transform $\mathcal F$ extends uniquely to a unitary operator \mbox{$\mathcal F:\mathcal H \to L^2(\V)$} such that $\mathcal F  T = M \mathcal F$,
where {$M:L^2(\V) \to L^2(\V)$} is the unbounded multiplication operator.
\end{thm}

\begin{rem} \label{rem:JacobiTransform}
In case $\al = \be$ the spectral decomposition of $T$ can be described using the Jacobi function transform \cite{K}. To see this, we apply the change of variable $x = \tanh(t)$, then the second-order differential operator $T$ defined by \eqref{eq:T} turns into
\[
\widehat T = \frac{d^2}{dt^2} + [\be-\al-(\al+\be+2)\tanh(t)]\frac{d}{dt} +  \frac{\ka^2-(\al+\be+3)^2}{4\cosh^2(t)}.
\]
For $\al=\be$, let $f_\la$ be a solution of the eigenvalue equation $\widehat T f_\la = \la f_\la$. Now define $F_\la^\pm (t) = \cosh^{\frac12(2\al+3\pm \ka)}(t) f_\la(t)$, then $F_\la$ satisfies
\[
\frac{d^2 F_\la^\pm}{dt^2} + (1\pm \ka) \tanh(t) \frac{ dF_\la^\pm}{dt}  = \Big(\la + (\al-1)^2 - \frac14(1\pm \ka)^2\Big) F_\la^\pm.
\]
Using the differential equation for Jacobi functions, see \cite[(1.1)]{K}, we now see that the
spectral decomposition of $T$ can be given using the Jacobi function transforms corresponding to
the Jacobi functions $\phi_{\de_\la}^{(-\frac12,\frac12\ka)}$ and
$\phi_{\de_\la}^{(-\frac12,-\frac12\ka)}$.
\end{rem}

We have an explicit inverse of the integral transform $\mathcal F$. Define for $x \in (-1,1)$ the
integral transform $\mathcal G$ by
\[
\begin{split}
(\mathcal Gf)(x) &=\frac{1}{2\pi D} \int_{\Om_2} \big(\, \varphi_\la^+(x) \ \varphi_\la^-(x)\,\big)
V(\la) f(\la) \frac{d\la}{-i\eta_\la} \\
& \quad + \frac{1}{2\pi D} \int_{\Om_1} f(\la) \varphi_\la(x)v(\la) \frac{d\la}{-i\de_\la}
+ \frac{1}{D} \sum_{\la \in \Om_d} f(\la) \varphi_{\la}(x) N_{\la}
\end{split}
\]
for all functions $f \in L^2(\V)$ for which the above integrals converge.

\begin{thm} \label{thm:integraltransform2}
The integral transform $\mathcal G$ extends uniquely to an operator $\mathcal G: L^2(\V) \to \mathcal H$ such that $\mathcal G = \mathcal F^{-1}$.
\end{thm}

Theorem \ref{thm:integraltransform} and \ref{thm:integraltransform2} are proved in Section
\ref{sec:proofs}. The following orthogonality relations are a result of
Theorem \ref{thm:integraltransform} by considering the discrete spectrum of $T$.

\begin{cor}
Let $\ka\geq 1$, then the following orthogonality relations hold
\[
\begin{split}
\int_{-1}^1 & \rFs{2}{1}{-m,\kappa-m}{1+\eta(\la_m)}{\frac{1+x}{2}} \rFs{2}{1}{-n,\kappa-n}{1+\eta(\la_n)}{\frac{1+x}{2}} \\
& \quad \times (1-x)^{\frac12(\de(\la_m)+\de(\la_n)-2)}(1+x)^{\frac12(\eta(\la_m)+\eta(\la_n)-2)} \,dx \\
& =  \de_{mn} \frac{2^{\ka-n-m}}{N_{\la_n}},
\end{split}
\]
for all $n,m \in \N$ such that $n,m \leq \frac12(\ka-1)$.
\end{cor}

\section{Matrix-valued orthogonal polynomials} \label{sec:matrixvaluedpol}
In this section we show that the differential operator $T$ can be realized as a five-diagonal
operator with respect to an orthonormal basis for $\mathcal H$. Using the spectral decomposition
for $T$ this leads to orthogonality relations for $2\times 2$-matrix-valued orthogonal polynomials.

\subsection{The five-diagonal operator}
The Jacobi polynomials are defined by
\[
P_n^{(\al,\be)}(x) = \frac{(\al+1)_n}{n!} \rFs{2}{1}{-n,n+\al+\be+1}{\al+1}{\frac{1-x}{2}}.
\]
For $\al,\be>-1$ they form an orthogonal basis for $\mathcal H$;
\begin{gather*}
\langle P_m^{(\al,\be)},P_n^{(\al,\be)}\rangle = \de_{mn} h_n^{(\al,\be)}, \qquad
h_n^{(\al,\be)} = \frac{\al+\be+1}{2n+\al+\be+1} \frac{(\al+1,\be+1)_n }{(\al+\be+1)_n n!}.
\end{gather*}
The Jacobi polynomials are eigenfunctions of the Jacobi differential operator
\begin{gather*}
L^{(\al,\be)}= (1-x^2) \frac{d^2}{dx^2} + [\be-\al-(\al+\be+2)x]\frac{d}{dx},\\
L^{(\al,\be)} P_n^{(\al,\be)} = -n(n+\al+\be+1)P_n^{(\al,\be)}.
\end{gather*}
We define $r(x)=1-x^2$, then for $x \in (-1,1)$ the polynomial $r$ can be written as
\begin{equation} \label{eq:r(x)}
r(x) = K \frac{w^{(\al+1,\be+1)}(x) }{w^{(\al,\be)}(x)}, \qquad K=  \frac{4(\al+1)(\be+1)}{ (\al+\be+2)(\al+\be+3)}.
\end{equation}
The differential operator $T=T^{(\al,\be;\ka)}$ defined by \eqref{eq:T} is related to $L^{(\al,\be)}$ by
\begin{equation} \label{eq:T=pL+ga}
T^{(\al,\be;\ka)} =M(r)\big( L^{(\al+1,\be+1)} +\rho\big), \qquad \rho =\frac14\left(\ka^2-(\al+\be+3)^2\right),
\end{equation}
where $M(r)$ denotes multiplication by $r$.

In \cite[Section 3.1]{GIK} it is shown that an operator of the form \eqref{eq:T=pL+ga} acts as a
five-term operator on a suitable basis of $\mathcal H$. In this case the basis consists of Jacobi
polynomials. We define $\phi_n = P_n^{(\al,\be)}/(h_n^{(\al,\be)})^{1/2}$, $n \in \N$, then
$\{\phi_n\}_{n \in \N}$ is an orthonormal basis for $\mathcal H^{(\al,\be)}$. We also define
$\Phi_n = P_n^{(\al+1,\be+1)}/(h_n^{(\al+1,\be+1)})^{1/2}$, $n \in \N$, then $\{\Phi_n\}_{n \in \N}$
is an orthonormal basis for $\mathcal H^{(\al+1,\be+1)}$. In order to write $T$ explicitly as a
five-diagonal operator on the basis $\{\phi_n\}_{n \in \N}$, we need a connection formula between
$\{\phi_n\}_{n \in \N}$ and $\{ \Phi_n\}_{n \in \N}$.
\begin{lem} \label{lem:connectionformula}
The following connection formula holds,
\[
\phi_n = \al_n \Phi_n + \be_n \Phi_{n-1}+ \ga_n \Phi_{n-2},
\]
where
\[
\begin{split}
\al_n  & = \frac{2}{\sqrt{K}} \frac{1}{2n+\al+\be+2}\sqrt{ \frac{ (\al+n+1)(\be+n+1)(n+\al+\be+1)(n+\al+\be+2)  }{ (\al+\be+2n+1)(\al+\be+2n+3)}},\\
\be_n &= (-1)^{n} \frac{2}{\sqrt{K}} \frac{(\be-\al) \sqrt{n(n+\al+\be+1)}}{(\al+\be+2n)(\al+\be+2n+2)},\\
\ga_n & = - \frac{2}{\sqrt{K}} \frac{1}{2n+\al+\be}  \sqrt{ \frac{n(n-1)(\al+n)(\be+n)}{(\al+\be+2n-1)(\al+\be+2n+1)} }.
\end{split}
\]
\end{lem}
\begin{proof}
There exists an expansion $\phi_n = \sum_{k = 0}^n a_{n,k} \Phi_k$, where
\[
a_{n,k} = \langle \phi_n ,\Phi_k \rangle_{\al+1,\be+1} = K^{-1} \langle \phi_n,r \Phi_k \rangle_{\al,\be}.
\]
Since $r$ has degree $2$, it follows from the orthogonality relations for $\phi_n$ that $a_{n,k}=0$ for $0\leq k \leq n-3$.

We compute the three remaining coefficients. The value of $a_{n,n}$ follows from comparing leading
coefficients; $a_{n,n} = \frac{\lc(\phi_n)}{\lc(\Phi_n)}$. We have
\[
\lc(\phi_n) = 2^{-n} (n+\al+\be+1)_n \sqrt{\frac{2n+\al+\be+1}{\al+\be+1} \frac{(\al+\be+1)_{n} }{n! (\al+1,\be+1)_n}}
\]
and $\lc(\Phi_n)$ is obtained by replacing $(\al,\be)$ by $(\al+1,\be+1)$,
which leads to the result for $\al_n=a_{n,n}$.

For a polynomial $p$ of degree $n$, let $k(p)$ denote the coefficient of $(1-x)^{n-1}$, then
\[
a_{n,n-1} = \frac{ k(\phi_n) - a_{n,n} k(\Phi_n) }{\lc(\Phi_{n-1})}.
\]
We have
\[
k(\phi_n) = (-1)^{n+1} \lc(\phi_n) \frac{2 n(\al+n)}{\al+\be+2n}.
\]
and $k(\Phi_n)$ is obtained by replacing $(\al,\be)$ by $(\al+1,\be+1)$, which gives the expression for $\be_n=a_{n,n-1}$.

Finally, the expression for $\ga_n=a_{n,n-2}$ follows from
\[
a_{n,n-2} = K^{-1}\frac{\lc(r\Phi_{n-2})}{ \lc(\phi_n)} =-K^{-1} \frac{\lc(\Phi_{n-2})}{ \lc(\phi_n)}. \qedhere
\]
\end{proof}

We now use \cite[Lemma 3.1]{GIK} to write the differential operator $T$ as a five-diagonal operator.
\begin{prop} \label{prop:T5term}
The operator $T$ defined by \eqref{eq:T} acts as a five-diagonal operator on the basis $\{\phi_n\}_{n \in \N}$ of $\mathcal H$ by
\begin{equation} \label{eq:T5term}
T \phi_n = a_n \phi_{n+2} + b_n\phi_{n+1} +c_n \phi_n + b_{n-1} \phi_{n-1} + a_{n-2} \phi_{n-2},
\end{equation}
with coefficients given by
\begin{gather*}
a_n =  K \al_n \ga_{n+2} (\La_n +\rho), \qquad b_n=K \al_n \be_{n+1}(\La_n+\rho)+ K \be_n \ga_{n+1} (\La_{n+1}+\rho), \\
\qquad c_n= K \al_n^2 (\La_n+\rho) + K \be_n^2(\La_{n-1}+\rho) + K \ga_n^2 (\La_{n-2}+\rho),
\end{gather*}
where $\La_n = -n(n+\al+\be+3)$, $K$,
$\rho$ are given by \eqref{eq:r(x)}, \eqref{eq:T=pL+ga} and $\al_n,\be_n,\ga_n$ are as in Lemma \ref{lem:connectionformula}.
\end{prop}
One easily verifies that $a_{-1}=a_{-2}=b_{-1}=0$. Furthermore, we have the factorization
\[
\La_n+\rho = - \Big(n+\frac12(\al+\be+3+\ka)\Big)\Big(n+\frac12(\al+\be+3-\ka)\Big).
\]

\subsection{Matrix-valued orthogonal polynomials}
From Theorem \ref{thm:integraltransform} and Proposition \ref{prop:T5term} it follows that the functions $\mathcal F\phi_n$, $n \in \N$, satisfy the five-term recurrence relation
\begin{equation}
\begin{split} \label{eq:5termrecurrence}
\la& (\mathcal F \phi_n)(\la) = \\&a_n (\mathcal F\phi_{n+2})(\la) + b_n
(\mathcal F\phi_{n+1})(\la) +c_n (\mathcal F\phi_n)(\la) + b_{n-1} (\mathcal F\phi_{n-1})(\la) + a_{n-2} (\mathcal F\phi_{n-2})(\la),
\end{split}
\end{equation}
for almost all $\la \in \si$. Furthermore, the set $\{\mathcal F\phi_n\}_{n \in \N}$ is an
orthonormal basis for $L^2(\V)$. We can determine an explicit expression for $\mathcal F\phi_n$ in
terms of hypergeometric functions.
\begin{lem} \label{lem:seriesforFphin}
For $n \in \N$, let $F_n(\de,\eta)=F_n(\de,\eta;\al,\be,\ka)$ denote the series
\[
\begin{split}
F_n(\de,\eta) = D_n \sum_{l=0}^n & \frac{ (-n,n+\al+\be+1,\frac12(\al+\de+1)_l }{ l! (\al+1,\frac12(\al+\be+\eta+\de+2))_l }\\
 & \times \rFs{3}{2}{ \frac12(1+\de+\eta+\ka),\frac12(1+\de+\eta-\ka),\frac12(\be+\eta+1)}{1+\eta,\frac12(\al+\be+\eta+\de+2+2l)}{1}
\end{split}
\]
with
\[
D_n= \frac12 \sqrt{ \frac{ 2n + \al+\be+1 }{\al+\be+1} \frac{ (\al+\be+1, \al+1)_n }{n! (\be+1)_n } }
\frac{ \Ga(\al+\be+2,\frac12(\al+\de+1),\frac12(\be+\eta+1))}{\Ga(\al+1,\be+1,\frac12(\al+\be+\eta+\de+2))}.
\]
Then, for $\la \in \si$,
\[
(\mathcal F\phi_n)(\la) =
\begin{cases}
\begin{pmatrix}
F_n(\de_\la, \eta_\la) \\ F_n(\de_\la,-\eta_\la)
\end{pmatrix},
& \la \in \Om_2,\\
F_n(\de_\la,\eta(\la)), & \la \in \Om_1,\\
F_n(\de(\la),\eta(\la)), & \la \in \Om_d.
\end{cases}
\]
\end{lem}
The above $_3F_2$-series converges absolutely if $\Re(\al-\de+1+2l)>0$, which is the case if $\la \in \Om_1\cup \Om_2$. For $\la \in \Om_d$ the $_3F_2$-series terminates.
\begin{proof}
We compute
\[
\begin{split}
I_n = \int_{-1}^1 & \left(\frac{1-x}{2}\right)^{-\frac12(\al-\de+1)} \left(\frac{1+x}{2}\right)^{-\frac12(\be-\eta+1)} \hskip-.2truecm
\rFs{2}{1}{\frac12(1+\de+\eta+\ka),\frac12(1+\de+\eta-\ka)} {1+\eta}{\frac{1+x}{2}} \\
&\times \rFs{2}{1}{-n,n+\al+\be+1}{\al+1}{\frac{1-x}{2}} w^{(\al,\be)}(x)dx.
\end{split}
\]
Interchanging the order of summation and integration, and using the Beta-integral, we obtain
\[
\begin{split}
I_n &= C_n \sum_{l=0}^n \sum_{m=0}^\infty \frac{ (-n,n+\al+\be+1)_l }{2^l l! (\al+1)_l } \frac{ (\frac12(1+\de+\eta+\ka),\frac12(1+\de+\eta-\ka))_m }{2^m m! (1+\eta)_m }\\
& \qquad \qquad \times \int_{-1}^1 (1-x)^{\frac12(\al+\de-1)+l} (1+x)^{\frac12(\be+\eta-1)+m} dx \\
& = C_n' \sum_{l=0}^n \sum_{m=0}^\infty \frac{ (-n,n+\al+\be+1,\frac12(\al+\de+1)_l }{ l! (\al+1)_l } \\
& \qquad \qquad \times \frac{ (\frac12(1+\de+\eta+\ka),\frac12(1+\de+\eta-\ka),\frac12(\be+\eta+1))_m }{ m! (1+\eta)_m \, (\frac12(\al+\be+\eta+\de+2))_{l+m}},
\end{split}
\]
where
\[
C_n' = \frac12 \frac{ \Ga(\al+\be+2,\frac12(\al+\de+1),\frac12(\be+\eta+1))}{\Ga(\al+1,\be+1,\frac12(\al+\be+\eta+\de+2))}.
\]
Now the result follows from the explicit expressions for $\phi_n$ and the integral transform $\mathcal F$.
\end{proof}

From Lemma \ref{lem:seriesforFphin} it follows that
\[
F_0(\de,\eta) = D_0 \rFs{3}{2}{\frac12(1+\de+\eta+\ka),\frac12(1+\de+\eta-\ka),\frac12(\be+\eta+1)}{1+\eta, \frac12(\al+\be+\eta+\de+2)}{1},
\]
and
\begin{multline*}
F_1(\de,\eta)  = D_1 \Bigg[ \rFs{3}{2}{\frac12(1+\de+\eta+\ka),\frac12(1+\de+\eta-\ka),\frac12(\be+\eta+1)}{1+\eta, \frac12(\al+\be+\eta+\de+2)}{1} \\
- \frac{(\al+\be+2)(\al+\de+1)}{(\al+1)(\al+\be+\eta+\de+2)} \rFs{3}{2}{\frac12(1+\de+\eta+\ka),\frac12(1+\de+\eta-\ka),\frac12(\be+\eta+1)}{1+\eta, \frac12(\al+\be+\eta+\de+4)}{1} \Bigg].
\end{multline*}
These two functions and the five-term recurrence relation from \eqref{eq:5termrecurrence}
completely determine the functions $\mathcal F\phi_n$.\\

We define $2\times2$-matrix-valued orthogonal polynomials $P_n$, $n \in \N$, by the three-term recurrence relations
\begin{gather}
\la P_n(\la) = A_n P_{n+1}(\la) + B_n P_{n}(\la) + A_{n-1}^* P_{n-1}(\la),\qquad \la \in \si, \label{eq:3termrecPn}\\
A_n = \begin{pmatrix}
a_{2n} & 0 \\ b_{2n+1} & a_{2n+1}
\end{pmatrix}, \qquad
B_n = \begin{pmatrix}
c_{2n} & b_{2n} \\ b_{2n} & c_{2n+1}
\end{pmatrix}, \nonumber
\end{gather}
and the initial conditions $P_{-1}(\la) = 0$ and $P_0(\la) = I$. From the five-term recurrence relation \eqref{eq:5termrecurrence} we obtain, for $m \in \N$,
\begin{equation} \label{eq:phi->Pm}
\begin{split}
\begin{pmatrix} \mathcal F\phi_{2m}(\la) \\ \mathcal F\phi_{2m+1}(\la) \end{pmatrix} (\la) &=
P_m(\la) \begin{pmatrix} \mathcal F\phi_0(\la)\\ \mathcal F\phi_1(\la) \end{pmatrix} \qquad \text{if}\ \la \in \Om_1 \cup \Om_d,\\
\begin{pmatrix} \mathcal F\phi_{2m}(\la)^t \\ \mathcal F\phi_{2m+1}(\la)^t \end{pmatrix} (\la) &= P_m(\la) \begin{pmatrix} \mathcal F\phi_0(\la)^t\\ \mathcal F\phi_1(\la)^t \end{pmatrix} \qquad \text{if}\ \la \in \Om_2.
\end{split}
\end{equation}
The orthogonality relations for $\mathcal F \phi_n$, $n \in \N$, can now be reformulated as orthogonality relations for the matrix-valued polynomials $P_n$, see \cite[Theorem 2.1]{GIK}.
\begin{thm} \label{thm:matrixorthognality}
The $2\times2$-matrix-valued orthogonal polynomials $P_n$, $n \in \N$, defined by \eqref{eq:3termrecPn} satisfy the orthogonality relations
\[
\begin{split}
\de_{mn}I= &\frac{1}{2\pi D} \int_{\Om_2} P_m(\la) W_2(\la) P_n(\la)^* \frac{ d\la}{-i\eta_\la}\\
 &+ \frac{1}{2\pi D} \int_{\Om_1} P_m(\la) W_1(\la) P_n(\la)^* v(\la) \frac{ d\la}{-i\de_\la} + \frac{1}{D} \sum_{\la \in \Om_d}  P_m(\la) W_1(\la) P_n(\la)^* N_{\la}
\end{split}
\]
with
\begin{gather*}
W_1(\la) = \begin{pmatrix}
|(\mathcal F \phi_0)(\la)|^2
& (\mathcal F \phi_0)(\la)\overline{(\mathcal F \phi_1)(\la)} \\
\overline{(\mathcal F \phi_0)(\la)}(\mathcal F \phi_1)(\la)
& |(\mathcal F \phi_1)(\la)|^2
\end{pmatrix},\\
W_2(\la) = \begin{pmatrix}
\langle (\mathcal F \phi_0)(\la), (\mathcal F \phi_0)(\la) \rangle_{V(\la)} & \langle (\mathcal F \phi_0)(\la), (\mathcal F \phi_1)(\la) \rangle_{V(\la)} \\
\langle (\mathcal F \phi_1)(\la), (\mathcal F \phi_0)(\la) \rangle_{V(\la)} & \langle (\mathcal F \phi_1)(\la), (\mathcal F \phi_1)(\la) \rangle_{V(\la)}
\end{pmatrix},
\end{gather*}
where $\langle x,y \rangle_{V(\la)} = x^* V(\la) y$.
\end{thm}
\begin{rem}
In \cite[Proposition 3.6]{GIK} a $q$-analog of Theorem \ref{thm:matrixorthognality} is considered.
The functions $\phi_n$ in this case are the little $q$-Jacobi polynomials, and the integral
transform $\mathcal F$ is simply the integral transform corresponding to the continuous dual $q$-
Hahn polynomials. It would be very interesting to see if similar results can be obtained for other
$q$-analogs of the Jacobi polynomials, such as big $q$-Jacobi polynomials \cite{AA85}, Askey-Wilson
polynomials \cite{AW85} and Ruijsenaars' $R$-function \cite{R}.
\end{rem}

\subsection{The special case $\al=\be$}
We assume $\al=\be$, and for convenience we also assume $\Om_d = \emptyset$. In this case
$\Om_1=\emptyset$, and $\de_\la=\eta_\la$ for all $\la \in \Om_2$. The spectral decomposition of
$T$ can now be obtained in a different way.\\

The coefficient $b_n$ in the five-diagonal expression for $T$ vanishes, so $T$ reduces to
a tridiagonal operator or Jacobi operator. Explicitly,
\[
T\phi_n = a_n \phi_{n+2} + c_n \phi_n + a_{n-2} \phi_{n-2},
\]
with
\[
\begin{split}
a_n &= \frac{(n+\frac12(\al+\be+3+\ka))(n+\frac12(\al+\be+3-\ka))}{2n+2\al+3} \sqrt{ \frac{(n+2)(n+1)(n+2\al+1)(n+2\al+2)}{(2n+2\al+1)(2n+2\al+5)}},\\
c_n &= - \frac{(n+2\al+1)(n+2\al+2)(n+\frac12(\al+\be+3+\ka))(n+\frac12(\al+\be+3-\ka))}{(2n+2\al+1)(2n+2\al+3)} \\
& \quad - \frac{n(n-1)(n+\frac12(\al+\be-1+\ka))(n+\frac12(\al+\be-1-\ka))}{(2n+2\al-1)(2n+2\al+1)}.
\end{split}
\]
The spectral decomposition can be described with the help of the orthonormal Wilson polynomials \cite{Wi80,AAR}, which are defined by
\[
\begin{split}
W_n(x^2;a,b,c,d) & =
\sqrt{ \frac{(a+b,a+c,a+d)_n (a+b+c+d)_{2n} }{n!(b+c,b+d,c+d,n+a+b+c+d-1)_n }} \\
& \quad \times \rFs{4}{3}{-n,n+a+b+c+d-1,a+ix,a-ix}{a+b,a+c,a+d}{1}.
\end{split}
\]
If $a,b,c,d>0$ these polynomials are orthonormal with respect to an absolutely continuous measure on $(0,\infty)$.

For $m \in \N$ we define
\[
\begin{split}
W_m^\mathrm{e}(x) &= W_m\big((2x)^2;\tfrac12(\al+1),\tfrac12(\al+1),\tfrac14(1+\ka),\tfrac14(1-\ka)\big),\\
W_m^\mathrm{o}(x) &= W_m\big((2x)^2;\tfrac12(\al+1),\tfrac12(\al+1),\tfrac14(3+\ka),\tfrac14(3-\ka)\big),
\end{split}
\]
then we obtain from the three-term recurrence relation for the Wilson polynomials
\[
\begin{split}
-\big((\al+1)^2+x^2\big) W_m^\mathrm{e}(x)& = a_{2m} W_{m+1}^\mathrm{e}(x) + c_{2m}W_m^\mathrm{e}(x) + a_{2m-2} W_{m-1}^\mathrm{e}(x),\\
-\big((\al+1)^2+x^2\big) W_m^\mathrm{o}(x)& = a_{2m+1} W_{m+1}^\mathrm{o}(x) + c_{2m+1}W_m^\mathrm{o}(x) + a_{2m-1} W_{m-1}^\mathrm{o}(x).
\end{split}
\]
Let $\mu^\mathrm{e}$ and $\mu^\mathrm{o}$ denote the orthogonality measures for the Wilson
polynomials $W_m^\mathrm{e}$ and $W_m^\mathrm{o}$. The
unitary operator $U:\mathcal H \to L^2(\mu^\mathrm{e})\oplus L^2(\mu^\mathrm{o})$, given by
\begin{equation} \label{eq:U}
U \phi_n =
\begin{cases}
W_{m}^\mathrm{e}& \text{if}\ n=2m,\\
W_m^\mathrm{o} & \text{if}\ n=2m+1,
\end{cases}
\end{equation}
satisfies $U T = M U$, where $M$ is multiplication by $-((\al+1)^2+x^2)$. So $T$ indeed has continuous spectrum $(-\infty,-(\al+1)^2)=\Om_2$, with multiplicity $2$.\\

The Hilbert space $\mathcal H^{(\al,\al)}$ can also be split up in a natural way. From \eqref{eq:T}
we see that $T$ (with $\al=\be$) leaves invariant the subspaces of even/odd functions, so we can
split $\mathcal H$ accordingly into $\mathcal H^\mathrm{e}$ and $\mathcal H^\mathrm{o}$. The Jacobi
(Gegenbauer) polynomials $\phi_{2m}(x)$ are even polynomials, hence they form an orthonormal basis
for $\mathcal H^\mathrm{e}$, and by a quadratic transformation they can be transformed into
multiples of $P_m^{(\al,-\frac12)}(2x^2-1)$. Similarly, the odd polynomials $\phi_{2m+1}(x)$ form
an orthonormal basis for $\mathcal H^\mathrm{o}$, and they can transformed into multiples of
$xP_m^{(\al,\frac12)}(2x^2-1)$. Obviously there are similar transformations for the Jacobi
polynomials $\Phi_{2m}$ and $\Phi_{2m+1}$. Now the operator $T$ restricted to
$\mathcal H^\mathrm{e}$ or $\mathcal H^\mathrm{o}$ can be treated as in \cite[Section 3]{IK}. The
unitary operator $U$ is given in each case by a Jacobi function transform, see Remark
\ref{rem:JacobiTransform}, so that we obtain a special case of Koornwinder's formula \cite{K85}
stating that Jacobi polynomials are mapped to Wilson polynomials by the Jacobi function transform.
In this light, \eqref{eq:phi->Pm} can be considered as a matrix-analog of Koornwinder's formula.

\begin{rem}
There exists an extension of Koornwinder's formula on the level of Wilson polynomials \cite[Theorem 6.7]{Groen03}:
Wilson polynomials are mapped to Wilson polynomials by the Wilson function transform. It would be
interesting to see if there also exists a matrix-analog of this formula.
\end{rem}

\section{Proofs for Section \ref{sec:inttrans}} \label{sec:proofs}
In this section we perform the spectral analysis of the second-order differential operator $T$ defined by \eqref{eq:T}, considered as an unbounded operator on $\mathcal H$.

\subsection{Eigenfunctions} \label{ssec:eigenfunctions}
The eigenvalue equation $ T f_\la = \la f_\la$ is a second order differential equation with regular
singular points at $1,-1$ and $\infty$, so it has hypergeometric (i.e.~$_2F_1$) solutions. We first
determine these solutions.

We define for $\la \in \C \setminus \big(\Om_1\cup \Om_2\big)$ the functions
\[
\begin{split}
\phi^-_\la(x) & = \left(\frac{1-x}{2}\right)^{-\frac12(\al+\de(\la)+1)} \left(\frac{1+x}{2}\right)^{-\frac12(\be+\eta(\la)+1)} \\
& \quad \times \rFs{2}{1}{ \frac12(1-\de(\la)-\eta(\la)-\ka), \frac12(1-\de(\la)-\eta(\la)+\ka)}{1-\eta(\la)}{\frac{1+x}{2}},\\
\phi^+_\la(x) & = \left(\frac{1-x}{2}\right)^{-\frac12(\al+\de(\la)+1)} \left(\frac{1+x}{2}\right)^{-\frac12(\be-\eta(\la)+1)} \\ &
\quad \times \rFs{2}{1}{ \frac12(1-\de(\la)+\eta(\la)-\ka), \frac12(1-\de(\la)+\eta(\la)+\ka)}{1+\eta(\la)}{\frac{1+x}{2}}.
\end{split}
\]
and
\[
\begin{split}
\psi^-_\la(x) & =\left(\frac{1-x}{2}\right)^{-\frac12(\al+\de(\la)+1)} \left(\frac{1+x}{2}\right)^{-\frac12(\be+\eta(\la)+1)} \\
& \quad \times \rFs{2}{1}{ \frac12(1-\de(\la)-\eta(\la)-\ka), \frac12(1-\de(\la)-\eta(\la)+\ka)}{1-\de(\la)}{\frac{1-x}{2}},\\
\psi^+_\la(x) & =\left(\frac{1-x}{2}\right)^{-\frac12(\al-\de(\la)+1)} \left(\frac{1+x}{2}\right)^{-\frac12(\be+\eta(\la)+1)} \\
& \quad \times \rFs{2}{1}{ \frac12(1+\de(\la)-\eta(\la)-\ka), \frac12(1+\de(\la)-\eta(\la)+\ka)}{1+\de(\la)}{\frac{1-x}{2}},
\end{split}
\]
where $\de$ and $\eta$ are defined by \eqref{eq:dela and etala}. From Euler's transformation for
$_2F_1$-series it follows that $\phi^\pm_\la$ is invariant under $\de(\la) \mapsto -\de(\la)$;
similarly, $\psi^\pm_\la$ is invariant under $\eta(\la) \mapsto -\eta(\la)$. Note that
$\psi_\la^\pm(x)$ is obtained from $\phi_\la^\pm(x)$ by the substitution $(\al,\be,x) \mapsto (\be,\al,-x)$, and vice versa.
Note also that the ${}_2F_1$-series in $\phi^\pm_\la$ is summable at $x=1$ if $\Re(\de(\la))>0$, and
that the ${}_2F_1$-series in $\psi^\pm_\la$ is summable at $x=-1$ if $\Re(\eta(\la))>0$.

\begin{rem}
From here on we will just write $\de$ and $\eta$, instead of $\de(\la)$ and $\eta(\la)$.
\end{rem}

\begin{prop}
The functions $\phi_\la^\pm$, $\psi_\la^\pm$ are solutions of the eigenvalue equation $Tf=\la f$.
\end{prop}
\begin{proof}
Suppose $f$ is a solution of the eigenvalue equation $Tf=\la f$.
A calculation shows that if $f(x) = (1-x)^{-\frac12(\al+\de+1)}(1+x)^{-\frac12(\be+\eta+1)}\phi(x)$, then $\phi$ satisfies
\begin{multline*}
(1-x^2) \phi''(x)+[\de-\eta+(\de+\eta-2)x]\phi'(x)
-\frac14 (1-\eta -\de-\ka)(1-\eta -\de+\ka) \phi(x)=0.
\end{multline*}
Now set $t=\frac12(1+x)$, then
\begin{multline*}
t(1-t)\frac{d^2\phi}{dt^2} + \Big[(1-\eta)-\big(1+\frac12(1-\de-\eta-\ka)+\frac12(1-\de-\eta+\ka)\big)t\Big]\frac{d\phi}{dt} \\
- \frac14 (1-\eta -\de-\ka)(1-\eta -\de+\ka) \phi = 0.
\end{multline*}
This is the hypergeometric differential equation (see e.g.~\cite[Ch.2]{AAR}) with coefficients
\[
a=\frac12(1-\de-\eta-\ka), \qquad b=\frac12(1-\de-\eta+ \ka), \qquad c = 1-\eta.
\]
The $_2F_1$-functions in $\phi^-_\la, \psi^-_\la$ are well-known solutions of this differential
equation, so $\phi^-_\la, \psi^-_\la$ are solutions of the eigenvalue equation. The proof for
$\phi^+_\la$ and $\psi^+_\la$ is similar.
\end{proof}

For later references we need connection formulas for $\phi_\la^\pm$ and $\psi_\la^\pm$.
\begin{prop} \label{prop:connectionformulas}
For $c$ defined by \eqref{eq:c-function},
\begin{align}
\psi^\pm_\la(x) & = c(\eta;\pm\de) \phi^+_\la(x) + c(-\eta;\pm\de)\phi^-_\la(x), \label{eq:psi+=cphi+ cphi-}\\
\phi^\pm_\la(x) & = c(\de;\pm\eta) \psi^+_\la(x) + c(-\de;\pm\eta)\psi^-_\la(x). \label{eq:phi+=cpsi+ cpsi-}
\end{align}
\end{prop}
\begin{proof}
This follows from a three-term transformation for $_2F_1$-functions, see e.g.~\cite[(2.3.11)]{AAR}.
\end{proof}

The following identities for the $c$-function turn out to be useful.
\begin{lem}\label{lem:identitiesc-function}
The $c$-function defined by \eqref{eq:c-function} satisfies:
\begin{enumerate}[(i)]
\item $c(x;y) = -\frac{y}{x} c(-y;-x)$,
\item $c(x;y)c(-x;-y) - c(x;-y)c(-x;y)=-\frac{y}{x}$.
\end{enumerate}
\end{lem}

\begin{rem}
Using the reflection equation for the $\Ga$-function, Lemma \ref{lem:identitiesc-function}(ii) is
equivalent to the trigonometric identity
\begin{gather*}
\sin(\pi x)\sin(\pi y)
= \sin(\frac{\pi}{2}(y-x+\ka))\sin(\frac{\pi}{2}(y-x-\ka)) - \sin(\frac{\pi}{2}(-y-x+\ka))\sin(\frac{\pi}{2}(-y-x-\ka)),
\end{gather*}
and can also be proved in this way.
\end{rem}

\begin{proof}
The first identity follows from $\Ga(z+1)=z\Ga(z)$.

For the second identity we note that Proposition \ref{prop:connectionformulas} implies
\[
\begin{pmatrix}
c(\de;\eta) & c(-\de;\eta) \\ c(\de;-\eta) & c(-\de;-\eta)
\end{pmatrix}
=
\begin{pmatrix}
c(\eta;\de) & c(-\eta;\de) \\ c(\eta;-\de) & c(-\eta;-\de)
\end{pmatrix}^{-1},
\]
which in turn implies
\[
c(\de;\eta) = \frac{ c(-\eta;-\de) }{c(\eta;\de)c(-\eta;-\de) - c(\eta;-\de)c(-\eta;\de)}.
\]
Applying the first identity to the numerator then gives
\[
c(\eta;\de)c(-\eta;-\de) - c(\eta;-\de)c(-\eta;\de)=-\frac{\de}{\eta},
\]
which proves the second identity.
\end{proof}
We also need the behavior of the eigenfunctions near the endpoints $-1$ and $1$.
\begin{lem} \label{lem:asymptotics}
For $x \downarrow -1$ we have
\[
\phi^\pm_\la(x) =\left(\frac{1+x}{2}\right)^{-\frac12(\be \mp \eta + 1)}\Big(1+\mathcal O(1+x) \Big).
\]
For $x \uparrow 1$ we have
\[
\psi^\pm_\la(x) =\left(\frac{1-x}{2}\right)^{-\frac12(\al \mp \de + 1)}\Big(1+\mathcal O(1-x) \Big).
\]
\end{lem}
\begin{proof}
This is straightforward from the explicit expressions as $_2F_1$-series.
\end{proof}

\begin{rem} \label{rmk:lemasymptotics}
Observe that the function $|\phi^\pm_\la|^2 w^{(\al,\be)}$ is in $L^1(-1,0)$ if and only if
$\pm \Re(\eta)>0$. Furthermore,  $|\psi^\pm_\la|^2 w^{(\al,\be)}$ is in  $L^1(0,1)$ if and only if
$\pm \Re(\de)>0$.
\end{rem}

\subsection{Spectral analysis} \label{ssec:SpectralAnalysis}
We determine the spectrum and the spectral decomposition of $T$.\\

For functions $f,g$ that are differentiable at a point $x \in (-1,1)$ we define
\[
[f,g](x)= p(x)W(f,g)(x),
\]
where
\[
p(x) = C (1-x)^{\al+2}(1+x)^{\be+2} , \qquad C=2^{-\al-\be-1}\frac{\Ga(\al+\be+2)}{\Ga(\al+1,\be+1) },
\]
and $W(f,g)$ denotes the Wronskian
\[
W(f,g)(x)=f'(x)g(x)-f(x)g'(x).
\]
For $1\leq a<b\leq 1$ we denote by $\mathcal D(a,b)$ the subspace of $L^2\big((a,b),w^{(\al,\be)}(x)dx\big)$ consisting of functions $f$ such that
\begin{itemize}
\item $f$ is continuously differentiable on $(a,b)$
\item $f'$ is absolutely continuous on $(a,b)$
\item $Tf \in L^2\big((a,b),w^{(\al,\be)}(x)dx\big)$
\end{itemize}
Note that $\mathcal D(a,b)$ is dense in $L^2\big((a,b),w^{(\al,\be)}(x)dx\big)$.
\begin{lem} \label{lem:wronskian}
Let $1\leq a<b\leq 1$ and $f,g \in \mathcal D(a,b)$, then
\[
\int_{a}^b\left( (Tf)(x) \overline{g(x)} - f(x) \overline{ (Tg)(x) }\right) w^{(\al,\be)}(x)\, dx =[f,\bar g](b)-[f,\bar g](a).
\]
\end{lem}
\begin{proof}
We write the differential operator $T$ as
\[
T = (1-x)^{-\al}(1+x)^{-\be} \frac{d}{dx}\Big((1-x)^{\al+2}(1+x)^{\be+2} \frac{d}{dx}\Big) + \rho(1-x^2),
\]
then it follows that
\[
\begin{split}
\int_{a}^b&\left( (Tf)(x) \overline{g(x)} - f(x) \overline{ (Tg)(x) }\right) w^{(\al,\be)}(x)\, dx= \\
&\int_{a}^b \left[\overline{g(x)}\frac{d}{dx}\Big(p(x) f'(x)\Big) -  f(x)\overline{ \frac{d}{dx}\Big(p(x) g'(x)\Big)}\right]\,dx.
\end{split}
\]
Using integration by parts this is equal to
\[
\Big[p(x) f'(x)\overline{g(x)} - p(x)f(x)\overline{g'(x)} \Big]_{a}^b,
\]
which gives the result.
\end{proof}

Let $\mathcal D_0 \subset \mathcal H$ consist of the functions in $\mathcal D(-1,1)$ with support
on a compact interval in $(-1,1)$.
\begin{prop}
The densely defined operator $(T,\mathcal D_0)$ is symmetric.
\end{prop}
\begin{proof}
Clearly we have $\lim_{x \downarrow -1}[f,\bar g](x) = \lim_{x\uparrow 1}[f,\bar g](x)=0$ for $f,g \in \mathcal D_0$.
Then the result follows from Lemma \ref{lem:wronskian}.
\end{proof}

The function $x \mapsto [f,g](x)$, $x \in (-1,1)$, is constant if $f$ and $g$ are solutions of the
eigenvalue equations $Ty=\la y$. In the following lemma we determine the value of the constant in
case of the eigenfunctions $\psi_\la^\pm$ en $\phi_\la^\pm$.

\begin{lem}\label{lem:Wronskians}
For $\la \in \C \setminus (-\infty,-(\al+1)^2)$,
\[
[\phi^-_\la,\phi^+_{\la}] = -\eta D\quad \text{and} \quad [\psi^+_{\la},\phi^+_{\la}]= -\eta D\, c(-\eta;\de),
\]
where $D=2^{\al+\be+3}C$.
\end{lem}

Note that $[\psi_\la^-,\psi_\la^+]$ and $[\psi_\la^-,\phi_\la^-]$ can be obtained from
Lemma \ref{lem:Wronskians} using $(\al,\be,x) \mapsto (\be,\al,-x)$ and $(\de,\eta)\mapsto (-\de,-\eta)$, respectively.
\begin{proof}
We have
\[
[\phi^-_\la,\phi^+_{\la}] = C \lim_{x\downarrow -1} (1-x)^{\al+2} (1+x)^{\be+2}
\left(\frac{d\phi^-_\la}{dx}(x) \phi^+_{\la}(x) - \phi^-_\la(x)\frac{d\phi^+_{\la}}{dx}(x)\right).
\]
Using
\[
\frac{d\phi^\pm_\la}{dx}(x) = -\frac14(\be\mp \eta+1) \left(\frac{1+x}{2}\right)^{-\frac12(\be\mp \eta+1)-1}\Big(1+\mathcal O(1+x)\Big), \qquad x \downarrow -1,
\]
we find
\[
[\phi^-_\la,\phi^+_{\la}] = -\eta D
\]
Now from the connection formula \eqref{eq:psi+=cphi+ cphi-} we obtain
\[
[\psi^+_\la,\phi^+_{\la}]=c(-\eta;\de) [\phi^-_\la,\phi^+_{\la}] = -\eta D\, c(-\eta;\de). \qedhere
\]
\end{proof}
Let us mention that from the explicit formula \eqref{eq:c-function} for $c(-\eta;\de)$ and Lemma \ref{lem:Wronskians},
it follows that $[\psi^+_{\la},\phi^+_{\la}]=0$ if and only if $\la \in \C \setminus (-\infty,-(\al+1)^2)$ is a solution of
\begin{equation} \label{eq:c=0}
\frac12\big(1+\de(\la)+\eta(\la)\pm \ka\big) = -n,\qquad n \in \N,
\end{equation}
or equivalently, for some $n\in \N$,
\begin{equation}\label{eq:equationforlambda}
\sqrt{ \la + (\al+1)^2 } + \sqrt{\la+(\be+1)^2} = -2n-1 \mp \ka.
\end{equation}

\begin{prop}
The symmetric operator $(T,\mathcal D_0)$ has a unique self-adjoint extension.
\end{prop}

We denote the self-adjoint extension again by $T$.

\begin{proof}
By \cite[Thm. XIII.2.10]{DS} the adjoint of $(T,\mathcal D_0)$ is $(T,\mathcal D(-1,1))$, so the
deficiency spaces of  $(T,\mathcal D_0)$ consist of the solutions of the differential equations
$Tf=\pm if$ that are in $\mathcal H$, \cite[Cor. XIII.2.11]{DS}.
Let $f\in \mathcal H$ be a solution of $Tf=if$, then $f$ must
be a linear combination of $\phi_i^+$ and $\phi_i^-$, since these are linearly independent
solutions of this eigenvalue equation by Lemma \ref{lem:Wronskians}.
Note that $\Re(\eta(i))>0$, so by
Remark \ref{rmk:lemasymptotics} $\phi_i^-$ is not $L^2$ near $-1$, which implies that $f$ is a
multiple of $\phi_i^+$. In the same way it follows that $f$ is a multiple of $\psi^+_i$. But
$\phi_i^+$ and $\psi_i^+$ are linearly independent  by Lemma \ref{lem:Wronskians},
hence $f= 0$. In the same way it follows that
$f \in \mathcal H$ satisfying $Tf=-i f$ is the zero function. So $T$ has deficiency indices
$(0,0)$, which implies it has a unique self-adjoint extension.
\end{proof}

Assume $\la \in \C\setminus \R$. In this case $\Re(\eta),\Re(\de)>0$, so $\phi_\la^+ \in L^2\big((-1,0),w^{(\al,\be)}(x)dx\big)$ and \mbox{$\psi_\la^+ \in L^2\big((0,1),w^{(\al,\be)}(x)dx\big)$}. We define the Green kernel by
\[
K_\la(x,y) =
\begin{cases}
\displaystyle \frac{ \phi^+_\la(x) \psi^+_{\la}(y)}{[\psi^+_\la,\phi^+_{\la}]}, & x < y,\\ \\
\displaystyle \frac{ \phi^+_\la(y) \psi^+_{\la}(x)} {[\psi^+_\la,\phi^+_{\la}]}, & x > y,
\end{cases}
\]
then $K(\cdot,y) \in \mathcal H$ for any $y \in (-1,1)$. The Green kernel is useful for describing the resolvent operator $R_\la = (T-\la)^{-1}$.
\begin{lem}
For $\la \in \C\setminus\R$ the resolvent $R_\la$ is given by
\[
R_{\la} f = \Big(y \mapsto \langle f, \overline{K_\la(\cdot,y)}\rangle \Big), \qquad f \in \mathcal D_0.
\]
\end{lem}
\begin{proof}
First note that if $f \in \mathcal D_0$,
\[
\lim_{x \downarrow -1}[\phi^+_\la,f](x) = 0 = \lim_{x \uparrow 1} [\psi^+_\la,f](x).
\]
Now from Lemma \ref{lem:wronskian} we obtain
\[
\begin{split}
R_\la (T-\la) f(y) &= \frac{\psi_{\la}^+(y)}{[\psi_\la^+,\phi_{\la}^+]} \int_{-1}^y \big((T-\la)f\big)(x) \phi_\la^+(x) w^{(\al,\be)}(x)\, dx \\
& \quad + \frac{\phi_{\la}^+(y)}{[\psi_\la^+,\phi_{\la}^+]}\int_{y}^1 \big((T-\la)f\big)(x) \psi_{\la}^+(x)  w^{(\al,\be)}(x)\, dx \\
& = \frac{1}{[\psi_\la^+,\phi_{\la}^+]}\Big(\psi_{\la}^+(y)[f,\phi_{\la}^+](y) - \phi_\la^+(y)[f,\psi_\la^+](y)\Big),
\end{split}
\]
since $(T-\la)\phi_\la^+=0$ and $(T-\la)\psi_\la^+=0$.
Writing out the terms between brackets, we obtain
\[
R_\la (T-\la) f(y) = \frac{1}{[\psi_\la^+,\phi_{\la}^+]}\Big(-f(y)p(y)\psi_{\la}^+(y) \frac{d\phi_\la^+}{dx}(y) + f(y)p(y) \phi_\la^+(y) \frac{d\psi_{\la}^+}{dx}(y)  \Big) = f(y).
\qedhere
\]
\end{proof}

Now we can determine the spectral measure $E$ of $T$ by
\begin{equation} \label{eq:spectralmeasure}
\langle E(a,b)f,g \rangle = \lim_{\eps' \downarrow 0} \lim_{\eps\downarrow 0} \frac{1}{2\pi i}
\int_{a+\eps'}^{b-\eps'} \Big( \langle R_{\la+i\eps}f,g\rangle - \langle R_{\la-i\eps}f,g\rangle \Big) d\la,\quad f,g \in \mathcal D_0,
\end{equation}
see \cite[Thm. XII.2.10]{DS}.
We write
\begin{equation} \label{eq:resolvent}
\langle R_{\la}f,g\rangle = \iint\limits_{(x,y) \in \triangle} \Big(f(x)\overline{g(y)} +
f(y)\overline{g(x)}\Big) \frac{ \phi_\la^+(x) \psi_{\la}^+(y) }{[\psi_\la^+,\phi_{\la}^+]} w^{(\al,\be)}(x) w^{(\al,\be)}(y)\, d(x,y),
\end{equation}
where $\triangle = \{ (x,y) \in \R^2 \mid -1<x<1,\ x<y<1 \}$.

Let $\la \in \R$. To compute the spectral measure we have to consider the limit
\begin{equation} \label{eq:limit}
\lim_{\eps \downarrow 0} \left(\frac{\phi_{\la-i\eps}^+(x)\psi_{\la-i\eps}^+(y) }{\eta(\la-i\eps) c(-\eta(\la-i\eps);\de(\la-i\eps))}-
\frac{ \phi_{\la+i\eps}^+(x) \psi_{\la+i\eps}^+(y) }{\eta(\la+i\eps) c(-\eta(\la+i\eps);\de(\la+i\eps))} \right).
\end{equation}
Note that
\[
\lim_{\eps \downarrow 0} \eta(\la+i\eps) =
\begin{cases}
\displaystyle{\lim_{\eps \downarrow 0} \eta(\la-i\eps) \in \R_{\geq 0}}, & \text{if}\ \la + (\be+1)^2  \geq 0,\\ \\
\displaystyle{\lim_{\eps \downarrow 0} \overline{\eta(\la-i\eps)} \in i\R_{> 0}}, & \text{if}\  \la + (\be+1)^2 <0,
\end{cases}
\]
and
\[
\lim_{\eps \downarrow 0} \de(\la+i\eps) =
\begin{cases}
\displaystyle{\lim_{\eps \downarrow 0} \de(\la-i\eps) \in \R_{\geq 0}}, & \text{if}\ \la + (\al+1)^2  \geq 0,\\ \\
\displaystyle{\lim_{\eps \downarrow 0} \overline{\de(\la-i\eps)} \in i\R_{> 0}}, & \text{if}\ \la + (\al+1)^2 <0.
\end{cases}
\]
We see that we have to distinguish several cases.

\subsection{The continuous spectrum}
We define an integral transform $\mathcal F_c^{(2)}$ mapping $f \in \mathcal D_0$ to a
$\C^2$-valued function on $\Om_2=(-\infty,-(\be+1)^2)$ by
\[
(\mathcal F_c^{(2)} f)(\la) = \int_{-1}^1 f(x) \begin{pmatrix} \varphi^+_\la(x)\\ \varphi^-_\la(x) \end{pmatrix} w^{(\al,\be)}(x) \,dx,\qquad \la \in \Om_2,\ f \in \mathcal D_0,
\]
where the functions $\varphi_\la^\pm$ are defined by \eqref{def:phi+-}.

\begin{prop} \label{prop:doublecontinuousspectrum}
Let $a,b \in \Om_2$ with $a<b$, then
\[
\langle E(a,b) f, g\rangle = \frac{1}{2\pi D} \int_a^b (\mathcal F_c^{(2)} g)(\la)^*
\begin{pmatrix} 1 & v_{12}(\la) \\ v_{21}(\la) & 1\end{pmatrix}
(\mathcal F_c^{(2)} f)(\la) \frac{ d\la}{-i\eta_\la},
\]
where (recall from \eqref{eq:v(2)})
\[
v_{21}(\la) = \frac{ c(\eta_\la;\delta_\la) }{c(-\eta_\la;\delta_\la)},
\]
and $v_{12}(\la) = \overline{v_{21}(\la)}$. Here $\de_\la$ and $\eta_\la$ are defined by \eqref{eq:dela and etala}.
\end{prop}
\begin{proof}
First observe that
$\lim_{\eps\downarrow 0}\de(\la+i\eps)= \de_\la$ and $\lim_{\eps \downarrow 0} \de(\la-i\eps) = \overline{\de_\la}=-\de_\la$.
Similarly, $\lim_{\eps\downarrow 0}\eta(\la+i\eps) = \eta_\la$ and $\lim_{\eps \downarrow 0} \eta(\la-i\eps) = \overline{\eta_\la}=-\eta_\la$. This gives us
\[
\lim_{\eps \downarrow 0} \phi_{\la+i\eps}^+ = \varphi_\la^+, \qquad \lim_{\eps \downarrow 0} \phi_{\la-i\eps}^+ = \varphi_\la^-.
\]
Now the limit in \eqref{eq:limit} is equal to
\[
I_\la(x,y)=\lim_{\eps\downarrow 0}\left( \frac{ \varphi^+_\la(x)
\psi^+_{\la+i\eps}(y)}{-\eta_\la c(-\eta_\la;\de_\la)} +\frac{\varphi^-_\la(x) \psi^+_{\la-i\eps}(y)} {-\eta_\la c(\eta_\la;-\de_\la)} \right).
\]
Using the connection formula \eqref{eq:psi+=cphi+ cphi-} this becomes
\[
\begin{split}
I_\la&(x,y)=\\ &-\frac1{\eta_\la}\Bigg(\frac{ c(\eta_\la;\delta_\la) }{c(-\eta_\la;\delta_\la) }\varphi_\la^+(x) \varphi_\la^+(y)
+  \varphi_\la^+(x)\varphi_\la^-(y) + \varphi_\la^-(x)\varphi_\la^+(y)+ \frac{ c(-\eta_\la;-\delta_\la) }{c(\eta_\la;-\delta_\la) } \varphi_\la^-(x)\varphi_\la^-(y) \Bigg),
\end{split}
\]
which is manifestly symmetric in $x$ and $y$. Using $\overline{\varphi_\la^+(x)}=\varphi_\la^-(x)$, we see that
\[
I_\la(x,y) = -\frac{1}{\eta_\la} \begin{pmatrix} \varphi_\la^+(y) \\ \varphi_\la^-(y) \end{pmatrix}^*
\begin{pmatrix} 1 & v_{12}(\la) \\ v_{21}(\la) & 1\end{pmatrix}
\begin{pmatrix} \varphi_\la^+(x) \\ \varphi_\la^-(x) \end{pmatrix}
\]
Now we can symmetrize the double integral in \eqref{eq:resolvent} again, and then the result
follows from \eqref{eq:spectralmeasure}.
\end{proof}

Next we define an integral transform $\mathcal F_c^{(1)}$ mapping $\mathcal D_0$ to complex-valued functions on $\Om_1=\big(-(\be+1)^2,-(\al+1)^2\big)$ by
\[
(\mathcal F_c^{(1)}f)(\la) = \int_{-1}^1 f(x) \varphi_\la(x) w^{(\al,\be)}(x)\,dx,\qquad \la \in \Om_1,\ f \in \mathcal D_0,
\]
where $\varphi_\la(x)$ is defined by \eqref{def:phi}.
\begin{prop}\label{prop:singlecontspectrum}
Let $a,b\in \Om_1$ with $a<b$, then
\[
\langle E(a,b)f,g\rangle = \frac{1}{2\pi D} \int_a^b (\mathcal F_c^{(1)}f)(\la) \overline{(\mathcal F_c^{(1)}g)(\la)}v(\la) \frac{ d\la}{-i\de_\la},
\]
where (recall from \eqref{eq:v(1)})
\[
v(\la) = \frac{1}{ c\big(\de_\la;\eta(\la)\big)c\big(-\de_\la;\eta(\la)\big) }.
\]
\end{prop}
\begin{proof}
In this case
\[
\lim_{\eps \downarrow 0} \de(\la+i\eps) =\de_\la = -\lim_{\eps\downarrow 0} \de(\la-i\eps) \quad \text{and}\quad
\lim_{\eps \downarrow 0}\eta(\la+i\eps) = \eta(\la) = \lim_{\eps\downarrow 0} \eta(\la-i\eps).
\]
Consequently, using Euler's transformation,
\[
\lim_{\eps \downarrow 0} \phi_{\la+i\eps}^+ = \varphi_\la = \lim_{\eps \downarrow 0} \phi_{\la-i\eps}^+,
\]
and
\[
\lim_{\eps \downarrow 0} \psi_{\la - i\eps}^+(x) = \lim_{\eps \downarrow 0} \psi_{\la+i\eps}^+  (x).
\]
so that the limit \eqref{eq:limit} is equal to
\[
I_\la(x,y) =\lim_{\eps \downarrow 0}\left( \frac{ \varphi_\la(x) \psi_{\la+i\eps}^-(y) }{-\de_\la c\big(\de_\la;\eta(\la)\big) }+
\frac{ \varphi_\la(x) \psi_{\la+i\eps}^+(y)}{-\de_\la c\big(-\de_\la;\eta(\la)\big)}\right),
\]
where we have used Lemma \ref{lem:identitiesc-function}(i). Using the connection formula \eqref{eq:phi+=cpsi+ cpsi-} we obtain
\[
\begin{split}
I_\la(x,y) & = \lim_{\eps \downarrow 0} \frac{ \varphi_\la(x)[c\big(-\de_\la;\eta(\la)\big) \psi_{\la+i\eps}^-(y)
+c\big(\de_\la;\eta(\la)\big)\psi_{\la+i\eps}^+(y)] }{-\de_\la c\big(\de_\la;\eta(\la)\big)c\big(-\de_\la;\eta(\la)\big) } \\
& = \frac{ \varphi_\la(x)\varphi_\la(y) }{-\de_\la c\big(\de_\la;\eta(\la)\big)c\big(-\de_\la;\eta(\la)\big) }.
\end{split}
\]
The result follows from \eqref{eq:spectralmeasure} and \eqref{eq:resolvent} after symmetrizing the double integral.
\end{proof}

Since the spectrum is a closed set, the points $-(\al+1)^2$ and $-(\be+1)^2$ must belong to the spectrum.
\begin{prop}
The points $-(\al+1)^2$ and $-(\be+1)^2$ belong to the continuous spectrum.
\end{prop}
\begin{proof}
This follows from the fact none of the eigenfunctions is in $\mathcal H$ for these values of $\la$, see
Remark \ref{rmk:lemasymptotics}.
\end{proof}

\subsection{The discrete spectrum}
Recall the finite set $\Om_d=\{ \la_n \mid n \in \N \text{ and } n\leq \frac12(\ka-1)\}$, where
$\la_n$ is defined by \eqref{def:la n}. For $\ka\geq 1$, i.e., if $\Om_d$ is nonempty, we define
the integral transform $\mathcal F_d$ on $\mathcal H$ by
\[
(\mathcal F_d f)(\la) = \langle f, \varphi_{\la} \rangle, \qquad \la \in \Om_d,\ f \in \mathcal H,
\]
where $\varphi_{\la_n}$ is defined by \eqref{def:phi la n}. Note that $\varphi_{\la_n}=\phi_{\la_n}^+$.
\begin{prop} \label{prop:discretespectrum}
Let $-(\al+1)^2<a<b$. If $\Om_d \cap (a,b)$ consists of exactly one number $\la_n$, then
\[
\langle E(a,b)f,g \rangle = (\mathcal F_df)(\la_n) \overline{(\mathcal F_dg)(\la_n)}\frac{ N_{\la_n}}{D},\qquad f,g\in \mathcal H,
\]
where (recall from \eqref{eq:N la n})
\[
N_{\la_n} = \Res{\la=\la_n}\left(\frac{c\big(\eta(\la),\de(\la)\big)}{\eta(\la)c\big(-\eta(\la);\de(\la)\big)}\right).
\]
Furthermore, if $\Om_d \cap (a,b)=\emptyset$, then
\[
\langle E(a,b)f,g\rangle =0, \qquad f,g \in \mathcal H.
\]
\end{prop}
\begin{proof}
Assume $\Om_d \cap (a,b)=\{\la_n\}$. By \eqref{eq:spectralmeasure} and \eqref{eq:resolvent} we have
\[
\begin{split}
\langle  E(a,b)f,g \rangle  = D^{-1}\iint\limits_{(x,y) \in \triangle } &\Big(f(x)\overline{g(y)} + f(y) \overline{g(x)}\Big) w^{(\al,\be)}(x) w^{(\al,\be)}(y) \\
& \times \Bigg[\frac1{2\pi i}\int_{\mathcal C} \frac{ \phi_{\la}^+(x) \psi_{\la}^+(y) }{-\eta(\la) c\big(-\eta(\la);\de(\la)\big)} d\la\Bigg] \, d(x,y),
\end{split}
\]
where $\mathcal C$ is a small clockwise oriented rectifiable closed curve encircling $\la_n$ exactly once.
The integral over the curve $\mathcal C$ is equal to
\[
\phi^+_{\la_n}(x) \psi_{\la_n}^+(y) \Res{\la=\la_n}\left(\frac{1}{\eta(\la) c\big(-\eta(\la);\de(\la)\big)}\right).
\]
By \eqref{eq:c=0} we have $c(-\eta(\la);\de(\la))=0$ if and only if $\la=\la_n$, $n \in \Z$,
where $\la_n$ is defined by \eqref{def:la n}.
So in this case \eqref{eq:psi+=cphi+ cphi-} becomes $\psi_{\la_n}^+ = c(\eta(\la_n);\de(\la_n))\phi_{\la_n}^+$,
from which we see that the integrand is symmetric in $x$ and $y$. Symmetrizing the double integral gives the result.
\end{proof}

\begin{cor} \label{cor:orthogonality}
Suppose $\Om_d$ is nonempty, then the following orthogonality relations hold:
\[
\langle \varphi_{\la_m},\varphi_{\la_n} \rangle = \frac{ D}{N_{\la_n}} \de_{mn}, \qquad \la_m,\la_n \in \Om_d.
\]
\end{cor}
\begin{proof}
Let $\la_n,\la_m \in \Om_d$ and set $f = \varphi_{\la_m}$ and $g = \varphi_{\la_n}$. From Proposition \ref{prop:discretespectrum} it follows that $\langle \varphi_{\la_m}, \varphi_{\la_n} \rangle =0$ if $\la_n \neq \la_m$. Furthermore, if $\la_n=\la_m$, then
\[
\langle \varphi_{\la_n},  \varphi_{\la_n} \rangle = \langle \varphi_{\la_n},  \varphi_{\la_n} \rangle \langle \varphi_{\la_n}, \varphi_{\la_n} \rangle \frac{N_{\la_n}}{D},
\]
from which the result follows.
\end{proof}

We have now completely determined the spectrum of $T$.
\begin{thm}
The self-adjoint closure of the
densely defined operator $(T,\mathcal D_0)$ has continuous spectrum
$\big(-\infty,-(\al+1)^2\big]$ and (possibly empty) discrete spectrum $\Om_d$. The sets $\Om_2$ and
$\Om_1$ inside the continuous spectrum have multiplicity two and one, respectively.
\end{thm}

\subsection{The integral transform} \label{ssec:inttrans}

We define an integral transform $\mathcal F$ on $\mathcal D_0$ by
\begin{equation} \label{eq:defF}
\mathcal Ff = \mathcal F_c^{(2)}f+\mathcal F_c^{(1)}f + \mathcal F_df, \qquad f \in \mathcal D_0.
\end{equation}
For $f \in \mathcal D_0$ this coincides with \eqref{def:F}.
\begin{prop} \label{prop:plancherel}
$\mathcal F$ extends uniquely to an isometry $\mathcal F: \mathcal H \to L^2(\V)$.
\end{prop}
\begin{proof}
For $f,g\in\mathcal D_0$ we have
\[
\langle f,g \rangle =\langle \mathcal Ff,\mathcal Fg \rangle_{\V}
\]
by Propositions \ref{prop:doublecontinuousspectrum}, \ref{prop:singlecontspectrum} and \ref{prop:discretespectrum},
so $\mathcal F:\mathcal D_0 \to L^2(\V)$ is an isometry. By continuity it extends uniquely to an isometry $\mathcal H \to L^2(\V)$.
\end{proof}
Our next goal is to show that $\mathcal F:\mathcal H\to L^2(\V)$ is surjective and determine the
inverse. For convenience we assume that $T$ has no discrete spectrum.

For $0<a<1$ we define
\[
\langle f,g \rangle_{a} = \int_{-a}^{a} f(x) g(x) w^{(\al,\be)}(x)\,dx,
\]
for all functions $f,g$ for which the integral converges. Note that for $f,g \in \mathcal H$ the
limit $a\uparrow 1$ gives the inner product $\langle f,\bar g\rangle$. Suppose now that $f_\la$ is
a solution of the eigenvalue equation $Tf_\la = \la f_\la$, then by Lemma \ref{lem:wronskian},
\[
\langle f_{\la}, f_{\la'} \rangle_{a} = \frac{ [f_{\la},f_{\la'}](a) - [f_{\la},f_{\la'}](-a) }{\la-\la'}, \qquad \la,\la' \in \R.
\]
We will use this expression with $f_\la=\varphi_\la^\pm$ and we want to let $a \uparrow 1$. We need to consider several cases.

\subsubsection{Case 1: $\la,\la'\in \Om_2$} From Lemma \ref{lem:asymptotics} we find for $x \downarrow -1$,
cf.~the proof of Lemma \ref{lem:Wronskians},
\[
\begin{split}
[\varphi_{\la}^+,\varphi_{\la'}^-](x)&=\frac{D}{2}(\eta_{\la}+\eta_{\la'}) \left( \frac{ 1+x }{2} \right)^{\frac12(\eta_{\la}-\eta_{\la'})}\Big(1+\mathcal O(1+x)\Big),\\
[\varphi_{\la}^+,\varphi_{\la'}^+](x)&=\frac{D}{2}(\eta_{\la}-\eta_{\la'}) \left( \frac{ 1+x }{2} \right)^{\frac12(\eta_{\la}+\eta_{\la'})}\Big(1+\mathcal O(1+x)\Big).
\end{split}
\]
The behavior at $x=-1$ of $[\varphi_\la^-,\varphi_{\la'}^+](x)$ and $[\varphi_\la^-,\varphi_{\la'}^-](x)$ follows from $\varphi_\la^+(x) = \overline{\varphi_\la^-(x)}$.
For $x \uparrow 1$ we use the expansion from \eqref{eq:phi+=cpsi+ cpsi-} (recall that $\varphi_\la^\pm = \lim_{\eps \downarrow 0} \phi_{\la+i\eps}^\pm$) and Lemma \ref{lem:asymptotics} to find
\[
[\varphi_{\la}^+,\varphi_{\la'}^-](x) = \frac{D}{2} \sum_{\epsilon,\epsilon' \in \{+,-\}} c(\epsilon \de_\la;\eta_\la) c(-\epsilon'\de_{\la'};-\eta_{\la'}) (\epsilon \de_\la + \epsilon'\de_{\la'})
\left(\frac{ 1-x}{2} \right)^{\frac12(\epsilon \de_\la - \epsilon' \de_{\la'})} \Big(1+ \mathcal O(1-x)\Big),
\]
and
\[
[\varphi_{\la}^+,\varphi_{\la'}^+](x)= \frac{D}{2} \sum_{\epsilon,\epsilon' \in \{+,-\}} c(\epsilon \de_\la;\eta_\la) c(-\epsilon'\de_{\la'};\eta_{\la'}) (\epsilon \de_\la + \epsilon'\de_{\la'})
\left(\frac{ 1-x}{2} \right)^{\frac12(\epsilon \de_\la - \epsilon' \de_{\la'})} \Big(1+ \mathcal O(1-x)\Big).
\]

We will need the following behavior of the $c$-functions.
\begin{lem} \label{lem:behavior cfunction}
The $c$-function defined by \eqref{eq:c-function} satisfies
\[
\begin{split}
c(\de_\la;\eta_\la) &=
\begin{cases}
\mathcal O(e^{-\pi \sqrt{-\la}}), & \la \to -\infty,\\
\mathcal O(1), & \la \uparrow -(\be+1)^2,
\end{cases} \\
c(-\de_\la;\eta_\la) &=
\begin{cases}
\mathcal O(1), & \la \to -\infty,\\
\mathcal O(1), & \la \uparrow -(\be+1)^2,
\end{cases}
\end{split}
\]
\end{lem}
\begin{proof}
This follows from the definition \eqref{eq:c-function} of the $c$-function, and well-known asymptotic properties of the $\Ga$-function, see e.g.~\cite[Section 1.4]{AAR}.
\end{proof}

\begin{prop} \label{prop:reproducingproperty1}
Let $f\in C(\Om_2)$ satisfy
\[
f(\la) =
\begin{cases}
\mathcal O(|\la|^{-1-\eps}), & \la \to -\infty,\\
\mathcal O(1), & \la \uparrow -(\be+1)^2,
\end{cases}
\]
for some $\eps>0$, and let $\la' \in \Om_2$, then
\[
\lim_{a \uparrow 1} \int_{\Om_2} f(\la) \langle \varphi_\la^+,\varphi_{\la'}^- \rangle_a\,  d\la =-2\pi i D \,\de_{\la'}c(-\de_{\la'};\eta_{\la'}) c(\de_{\la'};-\eta_{\la'})  f(\la').
\]
\end{prop}
\begin{proof}
Note that
\begin{equation} \label{eq:eta+eta'}
\frac{\eta_\la+\eta_{\la'}}{\la-\la'} = \frac{1}{\eta_\la-\eta_{\la'}}, \qquad \frac{\eta_\la-\eta_{\la'}}{\la-\la'} = \frac{1}{\eta_\la+\eta_{\la'}},
\end{equation}
and similar expressions are valid for $\de_\la$. Now use $\de_\la=i|\de_\la|$ and $\eta_\la=i|\eta_\la|$, and write $N=-\frac12\ln(\frac{1-a}{2})$, then
\[
\begin{split}
\lim_{a \uparrow 1}& \int_{\Om_2} f(\la)  \langle \varphi_\la^+,\varphi_\la^- \rangle_a \,d\la = \lim_{a \uparrow 1} \int_{\Om_2} f(\la)
\frac{[\varphi_{\la}^+,\varphi_{\la'}^-](a) - [\varphi_{\la}^+,\varphi_{\la'}^-](-a)}{\la-\la'} d\la \\
& = \frac{D}{2}\lim_{N \to \infty} \int_{\Om_2} f(\la) \sum_{\epsilon \in \{+,-\}}
\Big(\xi^\epsilon_-(\la) \frac{\cos N(|\de_{\la}|+\epsilon |\de_{\la'}|)}{\de_{\la}+\epsilon \de_{\la'}} + i\xi_+^\epsilon(\la)
\frac{\sin N(|\de_{\la}|+\epsilon |\de_{\la'}|)}{\de_{\la}+\epsilon \de_{\la'}} \Big)\, d\la,\\
& \quad - \frac{D}{2} \lim_{N \to \infty}  \int_{\Om_2} f(\la)\frac{ \cos N(|\eta_\la|-|\eta_{\la'}|) }{\eta_\la - \eta_{\la'}} d\la -
\frac{iD}{2} \lim_{N \to \infty}  \int_{\Om_2} f(\la)\frac{ \sin N(|\eta_\la|-|\eta_{\la'}|) }{\eta_\la - \eta_{\la'}} d\la,
\end{split}
\]
where
\[
\xi_\pm ^\epsilon(\la)= c(\de_\la;\eta_\la) c(\epsilon \de_{\la'};-\eta_{\la'}) \pm  c(-\de_\la;\eta_\la) c(-\epsilon \de_{\la'};-\eta_{\la'}).
\]
The terms with $\xi_\pm^+$ vanish by the Riemann-Lebesgue lemma, which follows from Lemma \ref{lem:behavior cfunction}
and the assumptions on $f$. \\
\textit{Claim}:
\[
\lim_{N \to \infty} \int_{\Om_2} f(\la) \xi_-^-(\la) \frac{ \cos N(|\de_\la|-|\de_{\la'}|) }{\de_\la - \de_{\la'}} d\la -
\lim_{N \to \infty} \int_{\Om_2} f(\la)  \frac{ \cos N(|\eta_\la|-|\eta_{\la'}|) }{\eta_\la - \eta_{\la'}} d\la =0.
\]
\begin{proof}[Proof of claim]
Using \eqref{eq:eta+eta'} and $\cos\te_1-\cos\te_2 = -2\sin(\frac{\te_1+\te_2}{2}) \sin(\frac{\te_1-\te_2}{2})$, we obtain
\[
\begin{split}
\xi_-^-(\la) \frac{ \cos N(|\de_\la|-|\de_{\la'}|) }{\de_\la - \de_{\la'}} -& \frac{ \cos N(|\eta_\la|-|\eta_{\la'}|) }{\eta_\la - \eta_{\la'}} =
\frac{\xi_-^-(\la)(\de_\la+\de_{\la'})-(\eta_\la+\eta_{\la'})}{\la-\la'}\cos N(|\de_\la|-|\de_{\la'}|)\\
&+ \frac{2 \sin\frac{N}{2}(|\de_\la|+|\eta_\la|-|\de_{\la'}|-|\eta_{\la'}|) \sin\frac{N}{2}(|\de_\la|-|\eta_\la|+|\de_{\la'}|-|\eta_{\la'}|) }{\eta_\la-\eta_{\la'}} .
\end{split}
\]
We multiply the right hand side of the above identity by $f(\la)$ and integrate over $\la$. Since
the function $\la \mapsto \frac{\xi_-^-(\la)(\de_\la+\de_{\la'})-(\eta_\la+\eta_{\la'})}{\la-\la'}$
has a removable singularity by Lemma \ref{lem:identitiesc-function}(ii), it follows from the
Riemann-Lebesgue lemma that the first term vanishes as $N \to \infty$. For the second term we may use
\[
\left|\frac{\sin\frac{N}{2}(|\de_\la|+|\eta_\la|-|\de_{\la'}|-|\eta_{\la'}|) }{\eta_\la-\eta_{\la'}}\right|\leq B
\]
for $\la$ in a neighborhood of $\la'$ and for some $B>0$, then we see that we can apply the Riemann-Lebesgue lemma again, which proofs the claim.
\end{proof}
To finish the proof of the proposition we use
\begin{equation} \label{eq:Dirichlettypekernel}
\lim_{N \to \infty} \frac{1}{\pi} \int_A^B g(x) \frac{ \sin N(x-y)}{x-y} dx = g(y),
\end{equation}
if $g \in L^1(A,B)$ is continuous. Then
\[
\begin{split}
\lim_{a \uparrow 1} \int_{\Om_2} f(\la) \langle \varphi_\la^+,\varphi_{\la'}^- \rangle_a\,  d\la
&= \frac{iD}{2}\lim_{N \to \infty}  \int_{\Om_2} f(\la) \xi_+^-(\la) \frac{\sin N(|\de_\la|-|\de_{\la'}|)}{\de_\la-\de_{\la'}} d\la \\
&\qquad- \frac{iD}{2} \lim_{N \to \infty}  \int_{\Om_2} f(\la)\frac{ \sin N(|\eta_\la|-|\eta_{\la'}|) }{\eta_\la - \eta_{\la'}} d\la\\
&=- \pi i D\big( \de_{\la'} \xi_+^-(\la') +\eta_{\la'} \big) f(\la')\\
\end{split}
\]
provided $\xi_+^-\,f$ and $f$ are continuous functions in $L^1(\Om_2)$, which is indeed the case.
Here we used the substitutions $x=|\de_\la|$ and $x=|\eta_\la|$ before applying \eqref{eq:Dirichlettypekernel}; note that $\frac{dx}{d\la}=-\frac{1}{2x}$ in both cases. Finally,
applying Lemma \ref{lem:identitiesc-function}(ii) with $(x,y)=(\de_\la,\eta_\la)$ the last expression becomes
\[
-2\pi i D \,\de_{\la'}c(-\de_{\la'};\eta_{\la'}) c(\de_{\la'};-\eta_{\la'})  f(\la'),
\]
which finishes the proof.
\end{proof}
The following result is proved in the same way as Proposition \ref{prop:reproducingproperty1}.
\begin{prop}\label{prop:reproducingproperty2}
Let $f \in C(\Om_2)$ satisfy the same conditions as in Proposition \ref{prop:reproducingproperty1}
and let $\la' \in \Om_2$, then
\[
\lim_{a \uparrow 1} \int_{\Om_2} f(\la) \langle \varphi_\la^+,\varphi_{\la'}^+ \rangle_a\,  d\la =-2\pi i D \,\de_{\la'}c(\de_{\la'};\eta_{\la'}) c(-\de_{\la'};\eta_{\la'})  f(\la').
\]
\end{prop}

By combining Propositions \ref{prop:reproducingproperty1} and \ref{prop:reproducingproperty2} we obtain the following result.
\begin{prop} \label{prop:vector-reproducingproperty}
Let $f_1$ and $f_2$ satisfy the conditions from Proposition \ref{prop:reproducingproperty1}, then
\[
\mathcal F_c^{(2)} \left[ \frac{1}{2\pi D} \int_{\Om_2}
\begin{pmatrix}
\varphi_\la^+(x) \\ \varphi_\la^-(x)
\end{pmatrix}^*
\begin{pmatrix}
f_1(\la) \\ f_2(\la)
\end{pmatrix}
\, d\la \right](\la') =-i\de_{\la'} A(\la')
\begin{pmatrix}
f_1(\la') \\ f_2(\la'),
\end{pmatrix}
\]
where
\[
A(\la') =
\begin{pmatrix}
c(-\de_{\la'};\eta_{\la'}) c(\de_{\la'};-\eta_{\la'}) & c(\de_{\la'};\eta_{\la'}) c(-\de_{\la'};\eta_{\la'}) \\
c(\de_{\la'};-\eta_{\la'}) c(-\de_{\la'};-\eta_{\la'})& c(-\de_{\la'};\eta_{\la'}) c(\de_{\la'};-\eta_{\la'})
\end{pmatrix}.
\]
\end{prop}

\begin{proof}
Let $f_1$ and $f_2$ satisfy the conditions of Proposition \ref{prop:reproducingproperty1}, then
from this proposition and from applying Fubini's theorem we obtain
\[
\begin{split}
-i\de_{\la'}c(\de_{\la'};\eta_{\la'}) c(-\de_{\la'};-\eta_{\la'})  f_2(\la') &= \frac{1}{2\pi D} \lim_{a \uparrow 1} \int_{\Om_2} f_2(\la)
\int_{-a}^a \varphi_\la^+(x) \varphi_{\la'}^-(x) w^{(\al,\be)}(x)\,dx\,d\la \\
& = \int_{-1}^1 \left[ \frac{1}{2\pi D} \int_{\Om_2} f_2(\la) \overline{ \varphi_\la^-(x)}\, d\la \right] \varphi_{\la'}^-(x)w^{(\al,\be)}(x)\,dx.
\end{split}
\]
From Propositions \ref{prop:reproducingproperty1} and \ref{prop:reproducingproperty2} we find three similar identities, leading to
\[
\int_{-1}^1 \left[ \frac{1}{2\pi D} \int_{\Om_2}
\begin{pmatrix}
\varphi_\la^+(x) \\ \varphi_\la^-(x)
\end{pmatrix}^*
\begin{pmatrix}
f_1(\la) \\ f_2(\la)
\end{pmatrix}
\, d\la \right]
\begin{pmatrix}
\varphi_{\la'}^+(x) \\ \varphi_{\la'}^-(x)
\end{pmatrix}
w^{(\al,\be)}(x)\,dx =-i\de_{\la'} A(\la')
\begin{pmatrix}
f_1(\la') \\ f_2(\la')
\end{pmatrix},
\]
which is the desired result.
\end{proof}
We need the inverse of the matrix $A(\la)$ from Proposition \ref{prop:vector-reproducingproperty}.
\begin{lem} \label{lem:Ainverse}
For $\la \in \Om_2$, $A(\la)^{-1} = V(\la)$ with $V(\la)$ defined by \eqref{eq:V}
\end{lem}
\begin{proof}
We have
\[
\begin{split}
\det A(\la)&  = c(\de_\la;-\eta_\la) c(-\de_\la;\eta_\la) \Big( c(\de_\la;-\eta_\la) c(-\de_\la;\eta_\la) -c(\de_\la;\eta_\la)c(-\de_\la;-\eta_\la) \Big)\\
& = \frac{\eta_\la}{\de_\la} c(\de_\la;-\eta_\la) c(-\de_\la;\eta_\la),
\end{split}
\]
by Lemma \ref{lem:identitiesc-function}(ii). Now it is straightforward to compute the inverse of
$A$. The result then follows from the definition of $V(\la)$ and Lemma
\ref{lem:identitiesc-function}(i).
\end{proof}

Let $C_0(\Om_2;\C^2)$ denote the set of continuous $\C^2$-valued functions $g=\left(\begin{smallmatrix} g_1\\ g_2 \end{smallmatrix}\right)$ on $\Om_2$ satisfying
\[
g_j(\la) =
\begin{cases}
\mathcal O(|\la|^{-\frac12-\eps}), & \la \to -\infty,\\
\mathcal O(|\la+(\be+1)^2|^{\frac12}), & \la \uparrow -(\be+1)^2,
\end{cases}
\qquad \qquad j=1,2,
\]
for some $\eps>0$. For $g \in C_0(\Om_2;\C^2)$ we define the function $\mathcal G_c^{(2)}g$ by
\[
(\mathcal G_c^{(2)}g)(x) = \frac{1}{2\pi D} \int_{\Om_2}
\begin{pmatrix} \varphi_{\la}^+(x) \\ \varphi_\la^-(x) \end{pmatrix}^*
V(\la) g(\la) \frac{d\la}{-i\eta_\la}, \qquad x \in (-1,1).
\]
\begin{prop} \label{prop:F(2)inverseofG(2)}
Let $g \in C_0(\Om_2;\C^2)$ and $\la \in \Om_2$, then $(\mathcal F_c^{(2)} \mathcal G_c^{(2)} g)(\la) = g(\la)$.
\end{prop}
\begin{proof}
Let $g \in C_0(\Om_2;\C^2)$, define the $\C^2$-valued function $f$ by $f(\la)=\left(\begin{smallmatrix}f_1(\la)\\f_2(\la) \end{smallmatrix}\right) = \frac{1}{-i\de_\la} A(\la)^{-1}g(\la)$. Since
\[
v_{21}(\la)=
\begin{cases}
\mathcal O(e^{-\pi\sqrt{-\la}}), & \la \to -\infty,\\
\mathcal O(1), & \la \uparrow -(\be+1)^2,
\end{cases}
\]
by Lemma \ref{lem:behavior cfunction}, the functions $f_1$ and $f_2$ satisfy the conditions from
Proposition \ref{prop:reproducingproperty1}. Now Proposition \ref{prop:vector-reproducingproperty} shows that
\[
\mathcal F^{(2)}_c \left[ \frac{1}{2\pi D} \int_{\Om_2}
\begin{pmatrix}
\varphi_\la^+(x) \\ \varphi_\la^-(x)
\end{pmatrix}^*
\frac{1}{-i\de_\la}A(\la)^{-1} g(\la)
\, d\la \right]
(\la') = g(\la').
\]
From Lemma \ref{lem:Ainverse} we see that the term inside square brackets is exactly $(\mathcal G_c^{(2)}g)(x)$.
\end{proof}

\subsubsection{Case 2: $\la,\la' \in \Om_1$}
In this case,
\[
\lim_{x \downarrow -1}[\varphi_\la,\varphi_{\la'}](x) = 0,
\]
and for $x \uparrow 1$ we have
\begin{multline*}
[\varphi_\la,\varphi_{\la'}](x) = \\ \frac{D}{2} \sum_{\epsilon,\epsilon' \in \{+,-\}} c(\epsilon \de_\la;\eta(\la)) c(-\epsilon'\de_{\la'};\eta(\la')) (\epsilon \de_\la + \epsilon'\de_{\la'})
\left(\frac{ 1-x}{2} \right)^{\frac12(\epsilon \de_\la - \epsilon' \de_{\la'})} \Big(1+ \mathcal O(1-x)\Big).
\end{multline*}
We have the following behavior of the $c$-functions.
\begin{lem} \label{lem:behavior cfunction(1)}
The $c$-function defined by \eqref{eq:c-function} satisfies
\[
c(\pm\de_\la;\eta(\la)) =
\begin{cases}
\mathcal O(1), & \la \downarrow -(\be+1)^2,\\
\mathcal O(|\la+(\al+1)^2|^{-\frac12}), & \la \uparrow -(\al+1)^2.
\end{cases}
\]
\end{lem}

In the same way as in Proposition \ref{prop:reproducingproperty1} this leads to the following result.
\begin{prop} \label{prop:reproducingproperty3}
Let $f$ be a continuous function satisfying
\[
f(\la) =
\begin{cases}
\mathcal O(1), & \la \downarrow -(\be+1)^2,\\
\mathcal O(|\la+(\al+1)^2|^{-\frac12+\eps}), & \la \uparrow -(\al+1)^2,
\end{cases}
\]
for some $\eps>0$, and let $\la' \in \Om_1$, then
\[
\lim_{a \uparrow 1} \int_{\Om_1} f(\la) \langle \varphi_\la,\varphi_{\la'} \rangle_a\,  d\la =-\frac{2\pi i  D \,\de_{\la'} }{W^{(1)}(\la')} f(\la'),
\]
where (recall from \eqref{eq:v(1)}) $v(\la') = \big(c(\de_{\la'};\eta(\la')) c(-\de_{\la'};\eta(\la'))\big)^{-1}$.
\end{prop}
Note that
\[
v(\la) =
\begin{cases}
\mathcal O(1), & \la \downarrow -(\be+1)^2,\\
\mathcal O(|\la+(\al+1)^2|), & \la \uparrow -(\al+1)^2,
\end{cases}
\]
by Lemma \ref{lem:behavior cfunction(1)}.
Let $C_0(\Om_1)$ denote the set of continuous functions $g$ on $\Om_1$ satisfying
\[
g(\la) =
\begin{cases}
\mathcal O(1), & \la \downarrow -(\be+1)^2,\\
\mathcal O(|\la+(\al+1)^2|^{\eps}), & \la \uparrow -(\al+1)^2,
\end{cases}
\]
for some $\eps>0$. We define an integral transform $\mathcal G_c^{(1)}$ on $C_0(\Om_1)$ by
\[
(\mathcal G_c^{(1)}g)(x) = \frac{1}{2\pi D} \int_{\Om_1} g(\la) \varphi_\la(x) W^{(1)}(\la) \frac{d\la}{-i \de_\la}, \qquad x \in (-1,1),\ g \in C_0(\Om_1).
\]
Now similar as in Proposition \ref{prop:F(2)inverseofG(2)}, it follows from Proposition \ref{prop:reproducingproperty3} that $\mathcal F_c^{(1)}$ is a left-inverse of $\mathcal G_c^{(1)}$.
\begin{prop} \label{prop:F(1)inverseofG(1)}
For $g \in C_0(\Om_1)$ and $\la \in \Om_1$, we have $(\mathcal F_c^{(1)} \mathcal G_c^{(1)} g)(\la) = g(\la)$.
\end{prop}

\subsection{The integral transform $\mathcal G$} \label{ssec:inttransformG}
We define $\mathcal G$ on $C_0(\Om_1) \cup C_0(\Om_2;\C^2)$ by $\mathcal G = \mathcal G_c^{(1)} \oplus \mathcal G_c^{(2)}$. We will show that $\mathcal F$ is a left-inverse of $\mathcal G$. We
need the following result.
\begin{prop} \* \label{prop:F(12)G(21)=0}
\begin{enumerate}[(i)]
\item Let $\la \in \Om_1$ and $g \in C_0(\Om_1)$, then $(\mathcal F_c^{(2)} \mathcal G_c^{(1)} g)(\la) = \left(\begin{smallmatrix} 0 \\ 0 \end{smallmatrix}\right)$.
\item Let $\la \in \Om_2$ and $g \in C_0(\Om_2;\C^2)$, then $(\mathcal F_c^{(1)} \mathcal G_c^{(2)} g)(\la) = 0$.
\end{enumerate}
\end{prop}
\begin{proof}
Let $\la \in \Om_2$ and $\la'\in \Om_1$, then
\[
\lim_{x \downarrow -1}[\varphi_\la^{\pm}, \varphi_{\la'}](x) = 0,
\]
and for $x \uparrow 1$,
\[
\begin{split}
[\varphi_\la^\pm,&\varphi_{\la'}](x) = \\
&\frac{D}{2} \sum_{\epsilon,\epsilon' \in \{+,-\}} c(\epsilon \de_\la;\pm \eta_\la) c(-\epsilon'\de_{\la'};\eta(\la'))
(\epsilon \de_\la + \epsilon'\de_{\la'}) \left(\frac{ 1-x}{2} \right)^{\frac12(\epsilon \de_\la - \epsilon' \de_{\la'})} \Big(1+ \mathcal O(1-x)\Big).
\end{split}
\]
Similar as in the proof of Proposition \ref{prop:reproducingproperty1} it follows by application of the Riemann-Lebesgue lemma that
\[
\lim_{a \uparrow 1} \int_{\Om_2} f(\la)  \langle \varphi_\la^\pm ,\varphi_{\la'} \rangle_a \,d\la = \lim_{a \uparrow 1} \int_{\Om_2} f(\la)
\frac{[\varphi_{\la}^\pm ,\varphi_{\la'}](a) - [\varphi_{\la}^\pm,\varphi_{\la'}](-a)}{\la-\la'} d\la =0,
\]
for suitable functions $f$. As in Proposition \ref{prop:vector-reproducingproperty}
we obtain from this $(\mathcal F_c^{(2)} \mathcal G_c^{(1)} g)(\la) = \left(\begin{smallmatrix} 0 \\ 0 \end{smallmatrix}\right)$.

In the same way, it follows from
\[
\lim_{a \uparrow 1} \int_{\Om_1} f(\la)  \langle \varphi_{\la},\varphi_{\la'}^\pm , \rangle_a \,d\la = \lim_{a \uparrow 1} \int_{\Om_1} f(\la) \frac{[\varphi_{\la} ,\varphi_{\la'}^\pm](a) - [\varphi_{\la},\varphi_{\la'}^\pm](-a)}{\la-\la'} d\la =0,
\]
for suitable functions $f$, that $(\mathcal F_c^{(1)} \mathcal G_c^{(2)} g)(\la) = 0$.
\end{proof}

Combining Propositions \ref{prop:F(2)inverseofG(2)}, \ref{prop:F(1)inverseofG(1)} and \ref{prop:F(12)G(21)=0}
shows that $(\mathcal F\circ \mathcal G)g=g$ for $g \in C_0(\Om_1) \cup C_0(\Om_2;\C^2)$.
\begin{prop} \label{prop:G=Finv}
The integral transform $\mathcal G$ extends uniquely to an operator $\mathcal G:L^2(\V) \to \mathcal H$ such that $\mathcal G = \mathcal F^{-1}$.
\end{prop}
\begin{proof}
Let $g \in C_0(\Om_1) \cup C_0(\Om_2;\C^2)$, then
\[
\langle g,g \rangle_{\V} = \langle (\mathcal F \circ \mathcal G)g, (\mathcal F \circ \mathcal G)g \rangle_{\V} = \langle \mathcal Gg, \mathcal Gg \rangle,
\]
by Proposition \ref{prop:plancherel}. Since $C_0(\Om_1) \cup C_0(\Om_2;\C^2)$ is dense in $\mathcal H$,
$\mathcal G$ extends by continuity uniquely to an operator $\mathcal G:L^2(\V) \to \mathcal H$, and
$\mathcal F \circ \mathcal G$ extends to the identity operator on $\mathcal H$, hence $\mathcal G = \mathcal F^{-1}$.
\end{proof}

\begin{rem}
In case the discrete spectrum $\Om_d$ is nonempty the inverse of $\mathcal F$ is the extension of the operator $\mathcal G = \mathcal G^{(1)}_c \oplus \mathcal G^{(2)}_c \oplus \mathcal G_d$ with
\[
(\mathcal G_d g)(x) = \frac{1}{D} \sum_{\la \in \Om_d} g(\la) \varphi_\la(x) N_\la,\qquad x \in (-1,1),
\]
for any function $g:\Om_d \to \C$. The proof in this case is the same as in the case of empty
discrete spectrum.
\end{rem}

\end{document}